\documentclass[11pt]{amsart}
\usepackage{geometry}                % See geometry.pdf to learn the layout options. There are lots.
\usepackage{graphicx}
\usepackage{amssymb}
\usepackage{epstopdf}
\usepackage{amsmath}
\usepackage{color}
\usepackage{hyperref}
 \usepackage{tikz}
 
\usepackage[all]{xy}
\xyoption{matrix}
\xyoption{arrow}

\usetikzlibrary{arrows,decorations.pathmorphing,backgrounds,positioning,fit,petri}

 \tikzset{help lines/.style={step=#1cm,very thin, color=gray},
help lines/.default=.5} % draws a grid spaced #1 cm
\tikzset{thick grid/.style={step=#1cm,thick, color=gray},
thick grid/.default=1} % draws a grid spaced #1 cm

\newcommand{\aaa}{{\bf a}}
\newcommand{\bbb}{{\bf b}}

\newtheorem{thm}{Theorem}[section]
\newtheorem*{thm*}{Theorem}

\newtheorem{lem}[thm]{Lemma}
\newtheorem{cor}[thm]{Corollary}
\newtheorem{prop}[thm]{Proposition}
\newtheorem{conj}[thm]{Conjecture}

\newtheorem{innercustomthm}{Theorem}
\newenvironment{customthm}[1]
  {\renewcommand\theinnercustomthm{#1}\innercustomthm}
  {\endinnercustomthm}
  	
\newtheorem{innercustomcor}{Corollary}
\newenvironment{customcor}[1]
	{\renewcommand
	\theinnercustomcor{#1}\innercustomcor}
	{\endinnercustomcor}

\theoremstyle{definition}
\newtheorem{defn}[thm]{Definition}

\theoremstyle{remark}
\newtheorem{rem}[thm]{Remark}
 
\numberwithin{equation}{section}

\newcommand{\vs}[1]{\vskip .#1 cm} %enter amount of skip wanted at #1

\newcommand{\then}{\Rightarrow}

\newcommand{\ifff}{\Leftrightarrow}

\newcommand{\onto}{\twoheadrightarrow}

\DeclareMathOperator{\Hom}{Hom}%
\DeclareMathOperator{\Ext}{Ext}%
\DeclareMathOperator{\End}{End}%
\DeclareMathOperator{\undim}{\underline{dim}}
 
\newcommand{\field}[1]{\mathbb{#1}}
\newcommand{\ZZ}{\ensuremath{{\field{Z}}}}
\newcommand{\CC}{\ensuremath{{\field{C}}}}
\newcommand{\RR}{\ensuremath{{\field{R}}}}

\newcommand{\NN}{\ensuremath{{\field{N}}}}

\newcommand{\commentout}[1]{}

\newcommand{\cS}{\ensuremath{{\mathcal{S}}}}

\newcommand\vare{\varepsilon}

\title{Stability conditions for affine type A}

\author{P.J. Apruzzese}
\address{MB 0082, Brandeis University, Waltham, MA 02454}\email{papruzze@brandeis.edu}

\author{Kiyoshi Igusa}
\address{Department of Mathematics, Brandeis University, Waltham, MA 02454}\email{igusa@brandeis.edu}

\begin{document}

\maketitle

%\tableofcontents

\begin{abstract}
We construct maximal green sequences of maximal length for any affine quiver of type $A$. We determine which sets of modules (equivalently $c$-vectors) can occur in such sequences and, among these, which are given by a linear stability condition (also called a central charge). There is always at least one such maximal set which is linear. The proofs use representation theory and three kinds of diagrams shown in Figure \ref{Figure001}. Background material is reviewed with details presented in two separate papers \cite{PartI} and \cite{PartII}.
\end{abstract}

%\listoffigures

% Section:

%\newpage
%%%%%%%%%%%%%%%%%%%%%%%%%%
%
%                Section  {0introD}
%
%%%%%%%%%%%%%%%%%%%%%%%%%%
%\input{0introD}

\section*{Introduction} % sec 0

This paper addresses the question of linearity of maximal green sequences of maximal length. This question originates in a conjecture by Reineke [R] in which he asks for a linear stability condition on a Dynkin quiver which makes all indecomposable modules stable. Isomorphism classes of such modules are in bijection with the positive roots of the underlying root system. Reineke showed that the sequence of stable positive roots corresponding to the stable modules of a linear stability condition gives a quantum dilogarithm identity and he wanted that identity to have one term for every positive root. Yu Qiu [Q] has shown that, for every Dynkin quiver, there exists an orientation of the quiver and a linear stability condition given by a central charge which makes all indecomposable modules stable. This had already been done in type $A_n$ with straight orientation by Reineke [R]. So, [Q] dealt with quivers of other Dynkin types.

It is very easy to see that there are nonlinear stability conditions (called ``maximal green sequences'') which make all positive roots stable, namely take all indecomposable modules going, depending on sign convention, either from left to right (Reineke's sign convention) or from right to left (our sign convention) in the Auslander-Reiten quiver of the path algebra. More precisely, we order the indecomposable modules in such a way that, for every irreducible map $A\to B$, $B$ comes before $A$. Thus the question is mainly about the linearity of the stability condition.

In this paper we prove Reineke's original conjecture in type $A_n$ with any orientation and we give a complete resolution to the extension of this question to quivers of type $\widetilde A_{n-1}$. Since there are infinitely many positive roots in that case, the corresponding problem is to find maximal green sequences of maximal finite length and to determine which are linear. Our results are the following.

Recall that $\widetilde A_{a,b}$, for positive integers $a,b$, denotes a cyclic quiver with $a$ arrow going clockwise and $b$ arrows going counterclockwise. For example, there are, up to isomorphism, two quivers of type $\widetilde A_{3,2}$ which we denote: (See \eqref{eq: quiver of type tilde A23} for the sign notation.)
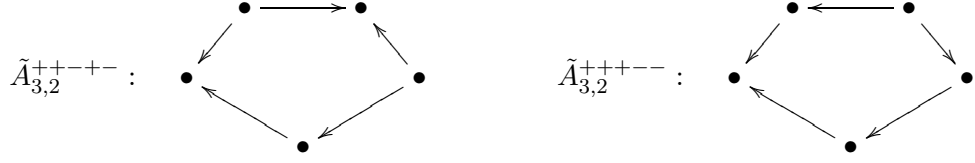
\[
% simple xy matrix
\xymatrixrowsep{10pt}\xymatrixcolsep{10pt}
\xymatrix{%begin xy matrix
&&&\bullet\ar[dl]\ar[rr] &&\bullet&&&&&&\bullet\ar[dl] &&\bullet\ar[rd]\ar[ll]\\
&\tilde A_{3,2}^{++-+-}:&\bullet &&&&\bullet\ar[lu]\ar[dll] && &\tilde A_{3,2}^{+++--}:& \bullet &&&&\bullet\ar[dll]\\
&&&&\bullet\ar[llu]&&&&&&&&\bullet\ar[llu]
	}%end xy matrix
\]

\begin{customthm}{M1}\label{thmA: maximum length of MGS}
Maximal green sequences for any quiver of type $\widetilde A_{a,b}$ have maximum length
\[
	\binom{a+b}2+ab.
\]
\end{customthm}

For $b=1$ this is already known \cite{Kase}.

Theorem \ref{thmA: maximum length of MGS} is shown in two steps: In Theorem \ref{thm M2b} we construct maximal green sequences of this length and in Theorem \ref{thm M2a} we show that there are no maximal green sequences of greater length. Although maximal green sequences were originally defined combinatorially \cite{Keller}, we use the language of representation theory which we review in \ref{ss11}. We take a fixed field $K$ and, for any acyclic quiver $Q$, we take the path algebra $\Lambda=KQ$ of $Q$. We use the ``wall-crossing'' definition of a maximal green sequence from \cite{PartI}, \cite{BST} which we review in Section \ref{sec3}. A {\bf maximal green sequence} for a finite dimensional algebra $\Lambda$ is a finite sequence of indecomposable modules $M_1,\cdots,M_m$ for which there exists a ``green path'' $\gamma$ going through the walls $D(M_1),\cdots,D(M_m)$ in that order and no other walls. The set of modules $M_i$ in this sequence are called the {\bf stable modules} of the sequence.

Maximal green sequences for $Q$ are in bijection with those for $KQ$ given by the wall crossing definition in Section \ref{sec3} and the dimension vectors of the stable module $M_i$ are the $c$-vectors of the corresponding combinatorially defined maximal green sequence. (See \cite{PartI}).

For quivers of Dynkin type, the longest green path passes through all of the walls and all indecomposable modules are stable. For quivers of type $\widetilde A_{a,b}$ there are infinitely many indecomposable modules and each maximal green sequence makes only finitely many of them stable. In this paper we determine all possible sets of stable modules of the maximum size given in Theorem \ref{thmA: maximum length of MGS} above. This is summarized by the following two theorems.

\begin{customthm}{M2}
For every quiver of type $\widetilde A_{a,b}$ with $(a,b)\neq (2,2)$ there are exactly $ab$ possible sets of stable modules for the maximal green sequences of length $\binom{a+b}2+ab$.
\end{customthm}

These sets are denoted $\cS_{k\ell}$ where $\vare_k=+,\vare_\ell=-$ and $0< k<\ell<k+n\le 2n$. See Definition \ref{def: Skell}. For $(a,b)\neq(2,2)$ these sets are distinct. (Theorem \ref{thm M2}, Corollary \ref{cor M2}.) Since there are $a$ choices for $k$ and $b$ choices for $\ell$, there are $ab$ such sets.

\begin{customthm}{M3}
For the quivers $\widetilde A_{2,2}^{+-+-}$, resp. $\widetilde A_{2,2}^{++--}$, there are $2$, resp. $3$, possible sets of stable modules for the maximal green sequences of the maximum length which is $10$.
\end{customthm}

For $\widetilde A_{2,2}^{+-+-}$, $\cS_{12}=\cS_{34}$ and $\cS_{14}=\cS_{36}$ and, for $\widetilde A_{2,2}^{++--}$, $\cS_{13}=\cS_{24}$. (Proposition \ref{prop: Skl are not distinct for (2,2)}).

Returning to the linearity question, we prove first that Reineke's conjecture holds for quivers of type $A_n$ with any orientation.

\begin{customthm}{L1}[Corollary \ref{Reineke's conjecture for An}]
For a quiver of type $A_n$ with any orientation, there exists a standard linear stability condition making all indecomposable modules stable.
\end{customthm}

For every linear stability condition $Z$ (also called a {\bf central charge} \ref{central charge}) there is a linear green path $\gamma_Z:\RR\to \RR^n$ given by $\gamma_Z(t)=t\bbb-\aaa$. If there are only finitely many stable modules, these modules form a maximal green sequence and we call the stability condition a {\bf finite linear stability condition}. 

In the affine case $\widetilde A_{a,b}$, we show that the following analogue of Reineke's conjecture hold.

\begin{customthm}{L2}
For any quiver of type $\widetilde A_{a,b}$ there exists a finite linear stability condition for which the number of stable modules is $\binom{a+b}2+ab$.
\end{customthm}

Theorem \ref{thm: nonexistence of Z} describes which of the sets $\cS_{k\ell}$ are {linear}, i.e., realized by linear stability conditions. For example, $\cS_{k\ell}$ is linear when $|k-\ell|\le 2$. Theorem \ref{thm: nonexistence of Z} implies the following.

\begin{customthm}{L3}
For the quiver $\widetilde A_{a,b}^\vare$, if either $a$ or $b$ is $\le2$, every possible set of stable modules of maximal size is linear. Otherwise (when $a,b\ge3$) there is at least one orientation of the quiver (choice of $\vare$) for which one of the sets of stable modules is not linear.
\end{customthm}

The smallest nonlinear example is $\widetilde A_{3,3}^{+++---}$ where 1 of the $ab=9$ sets of stable modules is nonlinear. And, in fact, this example is the cause of all nonlinearity. For every nonlinear example, $+++---$ will be a subsequence of the sign pattern $\vare$ up to cyclic order.

By the Deletion Lemma \ref{deletion lemma}, maximal green sequences for $\widetilde A_{a,b}^\vare$ restrict to maximal green sequences for $\widetilde A_{a',b'}^{\vare'}$ for many subsequence $\vare'$ of $\vare$ (in particular, $a'\le a$, $b'\le b$). By Remark \ref{rem: stable set Sk}, this includes the exceptional case where $b=0$. Although ``$\widetilde A_{n,0}$'' has infinite representation type (being an oriented cycle), only modules of length $\le n$ occur in a maximal green sequence since longer modules are not ``bricks'' (See Proposition \ref{prop: D(M) for bricks M}). In this paper we also exclude modules of length $n$. Equivalently, we mod out $rad^{n-1}$, and we have one of the well-known cluster-tilted algebras of type $D_n$ for $n\ge4$ considered in \cite{PartII}. In Theorem \ref{thm: upper bound for MGS of Dn} and Corollary \ref{cor L4, proved} we show the following.

\begin{customcor}{L4}\label{cor: Dn tilde case}
For $\Lambda_n=KQ_n/rad^{n-1}$ the path algebra of the oriented $n$-cycle $Q_n$ modulo the relation $rad^{n-1}=0$, the longest maximal green sequence has length $\binom{n}2+n-1$, there are $n$ different sets of stable modules of this size and all of them are linear.
\end{customcor}

\begin{figure}[htbp]
\begin{center}
\begin{tikzpicture}%[scale=3]
\begin{scope}[xshift=-5.5cm]
\clip (0.5,0) rectangle (4,5.5);
\draw[thick] (0,0) --(4,6) (1.5,0)--(3,4.5) (3,0)--(3,6);
\draw[thick] (0,4)--(4,0) (3,1)--(1,7);
\draw[thick,color=red] (2.12,1.88)--(2.23,3.34);
\draw (1,.8) node{\tiny$D(S_3)$};
\draw (1,3.5) node{\tiny$D(S_1)$};
\draw (3.4,2.7) node{\tiny$D(S_2)$};
\draw (2,.2) node{\tiny$D(P_3)$};
\draw (2.1,5) node{\tiny$D(P_2)$};
\draw[thick,->,color=green!80!black] (0,1.5)--(3.5,1.5);
\draw[color=green!80!black] (3.5,1.2) node{\small$\gamma$};
%\draw[fill] (3,4.5) circle[radius=2pt];
\end{scope}
\begin{scope}[yshift=2cm]
\draw[thick,color=red] (-.3,0)--(3,0);
\draw[fill,color=gray!20!white] (2,2)--(-.3,0)--(1,.5)--(2,2)--(3,0)--(1,.5);
\draw[thick] (2,2)--(-.3,0)--(1,.5)--(2,2)--(3,0)--(1,.5);\draw[fill] (-.3,0) circle[radius=2pt];
\draw[fill,color=blue] (2,2) circle[radius=2pt];
\draw[fill,color=red] (1,0.5) circle[radius=2pt];
\draw[fill] (3,0) circle[radius=2pt];
\draw (.6,0.15) node{\tiny$S_1$};
\draw (2,0.45) node{\tiny$P_3$};
\draw (1.7,1.1) node{\tiny$S_2$};
\draw (2.7,1) node{\tiny$S_3$};
\draw (.6,1.1) node{\tiny$P_2$};
\draw[color=red] (1.5,-0.2) node{\tiny$I_2$};
\draw (1.3,-1.5) node{$A^{-+}: 1\rightarrow 2\leftarrow 3$};
\end{scope}
\begin{scope}[xshift=5cm,yshift=1cm]
\draw[thick] (0,0)--(4,3.5);
\draw[thick] (0,3)--(4,0);
\draw[thick,color=red] (0,1.7)--(4,2.2);
\draw[thick,color=blue] (0,2.7)--(4,2);
\draw[fill] (.51,2.62) circle[radius=1.5pt] node[above]{\tiny$S_3$};
\draw[fill] (1.5,1.9) circle[radius=1.5pt] node[above]{\tiny$P_3$};
\draw[fill] (2.27,2) circle[radius=1.5pt] node[below]{\tiny$S_1$};
\draw[fill] (2.55,2.25) circle[radius=1.5pt] node[above]{\tiny$P_2$};
\draw[fill] (3.35,2.1) circle[radius=1.5pt] node[below]{\tiny$S_2$};
\draw[thick,color=red] (1.85,1.62) circle[radius=1.5pt] node[below]{\tiny$I_2$};
\end{scope}
\end{tikzpicture}
\caption{Three diagrams for the same linear stability condition on a quiver of type $A_3$ with stable modules $S_3,P_3,S_1,P_2,S_2$. Each diagrams show $I_2$ to be unstable: (1) $\gamma$ does not pass through the wall $D(I_2)$ (in red), (2) the chord $I_2$ is outside the polygon and (3) the vertex $I_2$ has a red line above it.}
\label{Figure001}
\end{center}
\end{figure}
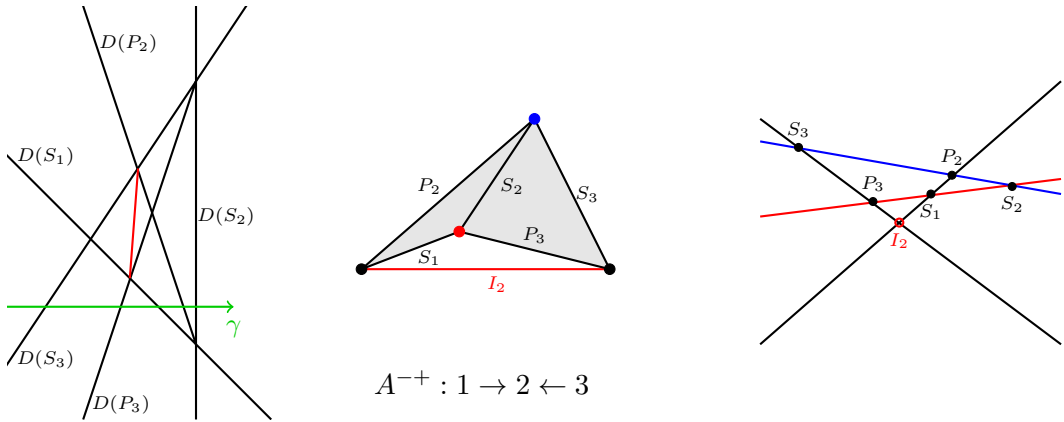

In order to prove these theorem we use the following three types of diagrams where (2) and (3) are always planar. (Figure \ref{Figure001} gives an example.)
\begin{enumerate}
\item Wall crossing diagrams. $M$ is stable when a green path $\gamma:\RR\to\RR^n$ passes through the interior of the wall $D(M)$. The stability condition given by $\gamma$ is linear when $\gamma$ is a straight line.
\item Chords in the ``stability polygon''. Certain chords represent stable modules in a linear stability condition. \item Wire diagrams. Stable modules are indicated by certain crossings of wires in the plane. When the wires are straight lines, this is a linear stability condition.
\end{enumerate}

In Section \ref{sec3} we use wall crossing diagrams to prove basic theorems about maximal green sequences, linear and nonlinear. Wire diagrams, introduced in Sections \ref{sec2}, \ref{sec4}, are used in Section \ref{sec5} to show that $\binom{a+b}2+ab$ is an upper bound for the maximum length of a maximal green sequence on any quiver of type $\tilde A_{a,b}$. Finally, in Section \ref{sec6}, chord diagrams, introduced in Section  \ref{sec1}, are used to realize this upper bound.

At the end of the paper (Sec \ref{sec7}) we give a summary of notation, definitions and constructions and a statement of the precise correspondence between there three kinds of diagrams.

%

% Section:

%\newpage
%%%%%%%%%%%%%%%%%%%%%%%%%%
%
%                Section  {1chordsD}
%
%%%%%%%%%%%%%%%%%%%%%%%%%%

\setcounter{section}0

\section{Chord diagrams for $A_n$}\label{sec1}

In this section we will review the representation theory of quivers, give the precise statement of Reineke's conjecture and a proof of the conjecture in type $A_n$ using chord diagrams.

\subsection{Representations of quivers}\label{ss11}

Suppose that $Q$ is a {\bf quiver}, i.e., a finite oriented graph, with vertex set $Q_0=\{1,2,\cdots,n\}$, arrow set $Q_1$ and no oriented cycles. An important case is when $Q$ is a linear quiver, in which case we say it has {\bf type $A_n$}. For example,
\[
	Q:1\xrightarrow\alpha 2\xrightarrow\beta 3\xleftarrow\gamma 4
\]
is a quiver of type $A_4$.

There are $2^{n-1}$ possible orientations for the arrows in a quiver of type $A_n$. We specify the orientation with a {\bf sign function} which we define to be any mapping
\[
	\vare:[0,n]=\{0,1,2,\cdots,n\}\to \{-,0,+\},
\]
written $\vare(i)=\vare_i$, so that $\vare_i=0$ iff $i=0$ or $n$. In the corresponding linear quiver, denoted $A_n^\vare$, the $i$th arrow points left $i\leftarrow i+1$ when $\vare_i=+$ and right $i\to i+1$ when $\vare_i=-$. Thus, the example above is $A_4^{--+}$ where we drop the values $v_0=v_n=0$ from the notation.

A {\bf representation} $M$ of a quiver $Q$ over a field $K$ is defined to be a sequence of finite dimensional vector spaces $M_i$, $i=1,\cdots,n$ and linear maps $M_a:M_i\to M_j$ for every arrow $a:i\to j$ in $Q$. The {\bf dimension vector} of $M$ is
\[
	\undim M:=(\dim_KM_1,\cdots,\dim_KM_n).
\]
The {\bf dimension} of $M$ is the dot product
\[
	\dim_KM:=\sum \dim_KM_i=(1,1,\cdots,1)\cdot \undim M.
\]

Recall that the {\bf positive roots} of the Dynkin diagram $A_n$ are the integer vectors
\[
	\beta_{ij}:=e_{i+1}+e_{i+2}+\cdots+e_j=(0,\cdots,0,1,1,\cdots,1,0,\cdots,0)
\]
with $1$s in positions $i+1,i+2,\cdots,j$ for any $0\le i<j\le n$. These are the dimension vectors of the indecomposable representations of $A_n^\vare$ for any $\vare$.

A representation of an acyclic quiver $Q$ is equivalent to a finitely generated modules over the {\bf path algebra} $\Lambda=KQ$. See, e.g., \cite{ASS}. A representation is {\bf indecomposable} if it is indecomposable as a $\Lambda$-module. We say $N$ is a {\bf subrepresentation} or {\bf submodule} of $M$ if, considered as $\Lambda$-modules, $N$ is a submodule of the module $M$. Equivalently, $N_i\subseteq M_i$ for all $i\in Q_0$ and $M_a(N_i)\subseteq N_j$ for all $a:i\to j$ in $Q_1$.

The following proposition is an easy exercise.

\begin{prop}
As representations of the quiver $A_n^\vare$, $M_{ij}$ is a subrepresentation of $M_{pq}$ if and only if the following three conditions are satisfied.
\begin{enumerate}
\item $p\le i<j\le q$
\item Either $p=i$ or $\vare_i=-$
\item Either $j=q$ or $\vare_j=+$
\end{enumerate}
\end{prop}

\subsection{Linear stability conditions}\label{central charge} We consider the dimension vectors of $\Lambda$-modules to be elements of $K_0\Lambda=\ZZ^n$. A {\bf central charge} on $\Lambda=KQ$ is defined to be an additive mapping
\[
	Z:K_0\Lambda\to\CC
\]
of the form
\[
	Z(x)=\aaa\cdot x+i\bbb\cdot x= r(x)e^{i\theta(x)}
\]
for fixed vectors $\aaa,\bbb\in\RR^n$ so that every coordinate $b_i$ of $\bbb$ is positive. We say that $Z$ is {\bf standard} if $\bbb=(1,1,\cdots,1)$. We also say that $\bbb$ is {\bf standard} in that case.

For a $\Lambda$-module $M$, the {\bf slope} of $M$ is then defined to be
\[
	\sigma_Z(M):=\frac{\aaa\cdot \undim M}{\bbb\cdot \undim M}=\cot\theta(M).
\]
In the standard case, $\bbb\cdot \undim M=\dim_KM$. The slope is undefined for $M=0$.

\begin{defn}\label{def: semistable}
A $\Lambda$-module $M$ is called $Z$-{\bf semistable} if \[
	\sigma_Z(M')\ge \sigma_Z(M)
\]
(equivalently, $\theta(M')\le \theta(M)$) for all nonzero submodules $M'\subseteq M$. $M$ is called $Z$-{\bf stable} if every proper submodule $M'\subset M$ has strictly larger slope: $\sigma_Z(M')> \sigma_Z(M)$.
\end{defn}

Because of this definition we often refer to $Z$ as a {\bf linear stability condition}. The problem is to determine the maximum finite number of stable modules given by a linear stability condition. It is an easy exercise to show that any $Z$-stable module is indecomposable.

Reineke's original conjecture states:

\begin{conj}\cite{R} For $Q$ a Dynkin quiver, there exist a standard linear stability condition making all indecomposable modules stable.
\end{conj}

Lutz Hille has claimed, privately to the second author, that this conjecture is not true in type $E_6$. Yu Qiu has shown \cite{Q} that this conjecture holds for at least one orientation of each Dynkin diagram if we drop the restriction that the linear stability condition should be \emph{standard}. We will prove the original conjecture for $A_n^\vare$ for any sign function $\vare$ using chord diagrams.

\subsection{Chord diagrams for type $A_n$}\label{dual vertices}

For a central charge $Z:K_0(KA_n^\vare)\to\CC$ given by $Z(x)=\aaa\cdot x+i\bbb\cdot x$, we will construct a ``stability polygon'' $C(Z)\subseteq \RR^2$ which will visually display which roots $\beta_{ij}$ are stable. This ``polygon'' might be degenerate, i.e., one-dimensional.

The {\bf vertices} of the stability polygon (also called {\bf dual vertices} of the quiver) are given by
\[
	p_i=(x_i,y_i):=(b_1+\cdots+b_i,a_1+\cdots+a_i)
\]
for $i=0,\cdots,n$. In particular $p_0=(0,0)$. For standard $Z$, $x_i=i$ for all $i$. The {\bf sign} of $p_i$ is $\vare_i$ and we sometimes write $p_k^+$ (or $p_\ell^-$) to mean \emph{$p_k$ which has sign $\vare_k=+$} (or: \emph{$p_\ell$ with $\vare_\ell=-$}). We say that the vertex $p_k$ is {\bf positive}, resp {\bf negative}, if $\vare_k=+$, resp $\vare_k=-$. {\bf Nonnegative} means either positive of equal to $p_0$ or $p_n$ which have sign $0$. {\bf Nonpositive} is similarly defined. 

\begin{lem} The slope of the line segment $V_{ij}:=\overline{p_ip_j}$ is equal to the slope of $M_{ij}$:
\[
	\text{slope }V_{ij}=\frac{a_{i+1}+\cdots+a_j}{b_{i+1}+\cdots+b_j}=\frac{\aaa\cdot\beta_{ij}}{\bbb\cdot\beta_{ij}}=\sigma_Z(M_{ij}).
\]
\end{lem}
The line segments $V_{ij}$ will be called {\bf chords} of the stability polygon $C(Z)$ defined below.

In the sequel we will use the words ``\emph{above}'' and ``\emph{below}'' to refer to relative position in the plane. Thus, a point $(x_0,y_0)$ is {\bf above}, resp. {\bf below}, a subset $S\subseteq \RR^2$ if $S$ contains a point $(x_0,z)$ with $y_0>z$, resp. $y_0<z$. We use ``\emph{higher}'' and ``\emph{lower}'' when referring only to the difference in the $y$-coordinates.

\begin{thm}\label{thm: stability of chords}
The module $M_{ij}$ is $Z$-semistable if and only if the following conditions hold.
\begin{enumerate}
\item For all $i<k<j$ with $\vare_k=+$, the point $p_k^+$ lies on or above the chord $V_{ij}$.
\item For all $i<\ell<j$ with $\vare_\ell=-$, the point $p^-_\ell$ lies on or below the chord $V_{ij}$.
\end{enumerate}
$M_{ij}$ is $Z$-stable if and only if it is $Z$-semistable and $p_i,p_j$ are the only vertices on $V_{ij}$.
\end{thm}

\begin{proof} We discuss only the stable case. The semistable case is similar.

($\then$) Suppose $M_{ij}$ is $Z$-stable. Since $M_{ik}$ and $M_{\ell j}$ are submodules of $M_{ij}$ for $\vare_k=+,\vare_\ell=-$ the slopes of the chords $V_{ik}$, $V_{\ell j}$ must be greater than the slope of $V_{ij}$. This holds if and only if $p_k$ is above $V_{ij}$ and $p_\ell$ is below $V_{ij}$. Thus (1) and (2) hold.

($\Leftarrow$) Given (1) and (2), we have seen that submodules of $M_{ij}$ of the form $M_{ik},M_{\ell j}$ will have slope greater than the slope of $M_{ij}$. The remaining indecomposable submodules are $M_{\ell k}$. But $V_{\ell k}$ has slope greater than that of $V_{ij}$ since it starts at the point $p_\ell^-$ below $V_{ij}$ and ends at $p_k^+$ above $V_{ij}$. Thus $M_{ij}$ is $Z$-stable. (See Figure \ref{Fig chords}.)
\end{proof}

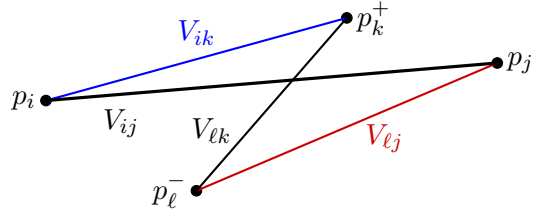
\begin{figure}[htbp]
\begin{center}
\begin{tikzpicture}%[scale=3]
\coordinate (I) at (0,1.2);
\coordinate (L) at (2,0);
\coordinate (K) at (4,2.3);
\coordinate (J) at (6,1.7);
\foreach \x in {I,L,K,J}\draw[fill] (\x) circle[radius=2pt];
\draw[thick, color=blue] (I)--(K) (2,2.1) node{$V_{ik}$};
\draw[thick, color=red!80!black] (L)--(J) (4.5,.7) node{$V_{\ell j}$};
\draw[very thick] (I)--(J) (1,.9) node{$V_{ij}$};
\draw[thick] (L)--(K) (2.2,.8) node{$V_{\ell k}$};
\draw (I) node[left]{$p_i$};
\draw (L) node[left]{$p_\ell^-$};
\draw (K) node[right]{$p_k^+$};
\draw (J) node[right]{$p_j$};
\end{tikzpicture}
\caption{The chords $V_{ik}$, $V_{\ell k}$ and $V_{\ell j}$ have slope greater than that of $V_{ij}$ if $p_k^+$ is above and $p_\ell^-$ is below the chord $V_{ij}$.}
\label{Fig chords}
\end{center}
\end{figure}

We will reformulate this theorem in terms of a polygon $C(Z)$ whose vertices are the points $p_i$, $i=0,\cdots,n$. The main property of $C(Z)$ will be that $M_{ij}$ is $Z$-semistable if and only if $V_{ij}\subseteq C(Z)$. We define the {\bf stability polygon} $C(Z)$ to be the intersection $C(Z)=C^+(Z)\cap C^-(Z)$ where $C^+(Z),C^-(Z)$ are defined below. Figure \ref{Fig: Polygon example} gives an example.

Let $P^+=\{p_{k}\,:\, \vare_{k}\ge0\}$, be the set of all nonnegative vertices. Number the elements of $P^+$ as $p_{k_i}$ where $0=k_0<k_1<\cdots<k_m=n$. In Figure \ref{Fig: Polygon example}, these are $k_i=0,1,3,8$. For every pair of consecutive elements $p_{k_{i-1}},p_{k_i}$, Let $C_i^+(Z)$ be union of the convex hull of the points $p_j$ for all $k_{i-1}\le j\le k_i$ with the set of all points below this convex hull. Let $C^+(Z)=\bigcup C^+_i(Z)$ be the union of these sets. In the example in Figure \ref{Fig: Polygon example}, $C^+(Z)$ is the union of three sets: $C^+_1(Z)$ is the chord $V_{01}$ and points below, $C^+_2(Z)$ is the chord $V_{13}$ and points below and $C^+_3(Z)$ is the union of the two chords $V_{34}$ and $V_{48}$ and the points below these two chords, i.e., $C^+(Z)$ is the blue curve and everything below the blue curve.

\begin{lem}\label{lem: property of C+(Z)}
For any $0\le i<j\le n$, the chord $V_{ij}$ lies in $C^+(Z)$ if and only if $p_k^+$ lies on or above $V_{ij}$ for all positive vertices $p_k^+$ between $p_i$ and $p_j$, i.e., so that $i<k<j$.
\end{lem}

\begin{proof}
($\then$) $C^+(Z)$ does not contain any of the points above a positive vertex $p_k^+$. Thus, if $p_k^+$ is below $V_{ij}$ then $V_{ij}$ cannot be contained in $C^+(Z)$. 

($\Leftarrow$) Suppose that each $p_k^+$ for $i<k<j$ lies on or above $V_{ij}$. Let $L$ be the piecewise linear curve going from $p_i$ to $p_j$, which goes through all positive vertices $p_k^+$ between $p_i$ and $p_j$ and which bends only at these positive vertices. We see that $L$, and thus all points below $L$, is contained in $C^+(Z)$. Since each vertex on $L$ lies on or above $V_{ij}$, the entire curve $L$ lies on or above $V_{ij}$. So, $V_{ij}$ is contained in $C^+(Z)$ as claimed.
\end{proof}

We define $C^-(Z)$ analogously to $C^+(Z)$: Let $P^-$ be the set of nonpositive vertices $p_\ell$. For any pair of consecutive vertices $p_{\ell_{i-1}},p_{\ell_i}$ in $P^-$, let $C_i^-(Z)$ be the union  of the convex hull of all vertices $p_j$ for $\ell_{i-1}\le j\le \ell_i$ and all points in $\RR^2$ above this convex hull. Let $C^-(Z)=\bigcup C^-_i(Z)$. Then we have the following Lemma analogous to Lemma \ref{lem: property of C+(Z)} above.

\begin{lem}\label{lem: property of C-(Z)}
For any $0\le i<j\le n$, the chord $V_{ij}$ lies in $C^-(Z)$ if and only if $p_\ell^-$ lies on or below $V_{ij}$ for all negative vertices $p_\ell^-$ so that $i<\ell<j$.\qed
\end{lem}

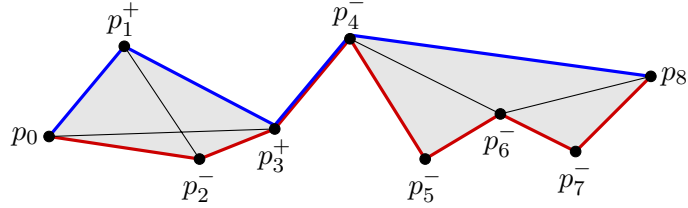
\begin{figure}[htbp]
\begin{center}
\begin{tikzpicture}%[scale=3]
\coordinate (p0) at (0,1.7);
\coordinate (p1) at (1,2.9);
\coordinate (p2) at (2,1.4);
\coordinate (p3) at (3,1.8);
\coordinate (p3u) at (3,1.85);
\coordinate (p4) at (4,3);
\coordinate (p4u) at (4,3.05);
\coordinate (p5) at (5,1.4);
\coordinate (p6) at (6,2);
\coordinate (p7) at (7,1.5);
\coordinate (p8) at (8,2.5);
\draw[fill,color=gray!20!white] (p3u)--(p1)--(p0)--(p2)--(p3);
\draw[fill,color=gray!20!white] (p4)--(p5)--(p6)--(p7)--(p8)--(p4u);
\draw (p0) node[left]{$p_0$};
\draw (p1) node[above]{$p_1^+$};
\draw (p2) node[below]{$p_2^-$};
\draw (p3) node[below]{$p_3^+$};
\draw (p4) node[above]{$p_4^-$};
\draw (p5) node[below]{$p_5^-$};
\draw (p6) node[below]{$p_6^-$};
\draw (p7) node[below]{$p_7^-$};
\draw (p8) node[right]{$p_8$};
\draw[very thick, color=red!80!black] (p0)--(p2)--(p3)--(p4)--(p5)--(p6)--(p7)--(p8);
\draw[very thick, color=blue] (p0)--(p1)--(p3u)--(p4u)--(p8);
\foreach\x in {p0,p1,p2,p3,p4,p5,p6,p7,p8} 
\draw[fill] (\x) circle [radius=2pt];
\draw (p0)--(p3) (p1)--(p2) (p4)--(p6)--(p8);
\end{tikzpicture}
\caption{The stability polygon $C(Z)=C^+(Z)\cap C^-(Z)$ consists of the two shaded regions and the chord $V_{34}=\overline{p_3p_4}$. The blue lines and points below the blue lines form $C^+(Z)$. The red lines and points above them form $C^-(Z)$. The 14 chords in $C(Z)$ represent the 14 $Z$-stable modules. Also, $M_{47}$ is $Z$-semistable since $p_6^-$ lies on the chord $V_{47}$.}
\label{Fig: Polygon example}
\end{center}
\end{figure}

\begin{thm}\label{thm: Mij Z-semistable iff Vij in C(Z)}
Given a central charge $Z:KA_n^\vare\to\CC$, the indecomposable $KA_n^\vare$-module $M_{ij}$ is $Z$-semistable if and only if the chord $V_{ij}$ lies in the stability polygon $C(Z)=C^+(Z)\cap C^-(Z)$. $M_{ij}$ is $Z$-stable if and only if it is $Z$-stable and $V_{ij}$ has no internal vertices.
\end{thm}

\begin{proof}
This follows immediately from Theorem \ref{thm: stability of chords} and Lemmas \ref{lem: property of C+(Z)}, \ref{lem: property of C-(Z)} above.
\end{proof}

\subsection{Reineke's conjecture}

We can now prove Reineke's conjecture for $A_n$ with any orientation.

\begin{cor}[Reineke's conjective for $A_n$]\label{Reineke's conjecture for An}
For a quiver $A_n^\vare$ with any sign function $\vare$, there is a standard linear stability condition making all indecomposable modules stable.
\end{cor}

\begin{proof}
By Theorem \ref{thm: Mij Z-semistable iff Vij in C(Z)} above it suffices to find a standard central charge $Z$ (with $\bbb=(1,1,\cdots,1)$) so that the stability polygon $C(Z)$ is convex and so that no three vertices are collinear. Such a central charge is given by inscribing the stability polygon $C(Z)$ in the circle of radius $n/2$ centered at $(n/2,0)$ and letting the positive vertices $p_k^+$ lie on the upper semi-circle and the negative $p_\ell^-$ lie on the lower semi-circle. And $p_0,p_n$ on the $x$-axis.

More precisely: $p_0=(0,0)$, $p_n=(n,0)$,
\[
	p_s^\delta=(s,\delta\sqrt{n^2/4-(s-n/2)^2})
\]
for all $0<s<n$. See Figure \ref{fig: C(Z) in circle}.
\end{proof}

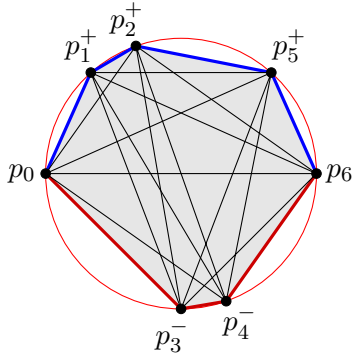
\begin{figure}[htbp]
\begin{center}
\begin{tikzpicture}[scale=.6]
\coordinate (p0) at (0,0);
\coordinate (p1) at (1,2.24);
\coordinate (p1L) at (.8,2.24);
\coordinate (p2) at (2,2.83);
\coordinate (p2L) at (1.8,2.83);
\coordinate (p3) at (3,-3);
\coordinate (p3L) at (2.8,-3);
\coordinate (p4) at (4,-2.83);
\coordinate (p4R) at (4.3,-2.7);
\coordinate (p5) at (5,2.24);
\coordinate (p5R) at (5.4,2.24);
\coordinate (p6) at (6,0);
\draw[color=red] (3,0) circle [radius=3cm];
\draw[fill,color=gray!20!white] (p0)--(p1)--(p2)--(p5)--(p6)--(p4)--(p3)--(p0);
\draw[very thick,color=blue] (p0)--(p1)--(p2)--(p5)--(p6);
\draw[very thick, color=red!80!black] (p6)--(p4)--(p3)--(p0);
\foreach\x in {p0,p1,p2,p3,p4,p5,p6} 
\draw[fill] (\x) circle [radius=3pt];
\draw (p0) node[left]{$p_0$};
\draw (p6) node[right]{$p_6$};
\draw (p1L) node[above]{$p_1^+$};
\draw (p2L) node[above]{$p_2^+$};
\draw (p5R) node[above]{$p_5^+$};
\draw (p3L) node[below]{$p_3^-$};
\draw (p4R) node[below]{$p_4^-$};
\draw (p6)--(p0)--(p4) (p2)--(p0)--(p5);
\draw (p2)--(p4)--(p5) (p5)--(p3)--(p6) (p3)--(p2)--(p6);
\draw (p6)--(p1)--(p4) (p3)--(p1)--(p5);
\end{tikzpicture}
\caption{Proof of Reineke's conjecture: The stability polygon $C(Z)$ is convex with no three vertices collinear. So, all chords $V_{ij}$ are stable.}
\label{fig: C(Z) in circle}
\end{center}
\end{figure}
%

%

% Section:

%\newpage
%%%%%%%%%%%%%%%%%%%%%%%%%%
%
%                Section  {2wiresD}
%
%%%%%%%%%%%%%%%%%%%%%%%%%%

\setcounter{section}1

\section{Linear wire diagrams for $A_n$}\label{sec2}

We briefly discuss the wire diagrams of type $A_n$. Given a central charge $Z(x)=\aaa\cdot x+i\bbb\cdot x$ for the quiver $A_n^\vare$, the {\bf wires} $L_i\subseteq \RR^2$, $i=0,\cdots,n$ are defined to be the graphs of the functions $f_i:\RR\to\RR$ given by:
\[
	f_i(t)= t(m-b_1-b_2-\cdots-b_i)+a_1+\cdots+a_i
\]
where $m$ is any convenient real number. The value of $m$ is not important since $f_j-f_i$ is independent of $m$. For every $0\le i<j\le n$ let $t_{ij}$ be the $x$-coordinate of the intersection point $L_i\cap L_j$, i.e., the solution of the equation $f_i(t_{ij})=f_j(t_{ij})$. This is:
\[
	t_{ij}=\frac{a_{i+1}+\cdots+a_j}{b_{i+1}+\cdots+b_j}=\sigma_Z(M_{ij})=\text{slope of }V_{ij}.
\]
We have the following easy theorem.

\begin{thm}\label{thm: stability for linear wires}
For $Z$ a central charge on $KA_n^\vare$,
$M_{ij}$ is $Z$-stable if and only if the following two conditions hold.
\begin{enumerate}
\item For all $i<k<j$ with $\vare_k=+$, $f_k(t_{ij})>f_i(t_{ij})=f_j(t_{ij})$. Equivalently, $L_k$ lies above the point $L_i\cap L_j$.
\item For all $i<\ell <j$ with $\vare_\ell=-$, $f_\ell(t_{ij})<f_i(t_{ij})=f_j(t_{ij})$, i.e., $L_\ell$ passes under $L_i\cap L_j$.
\end{enumerate}
\end{thm}

\begin{proof} This follows from Theorem \ref{thm: stability of chords} since (1) is equivalent to the condition that $p_k$ lies above the chord $V_{ij}$ and (2) is equivalent to the condition that $p_\ell$ lies below $V_{ij}$.

Proof of equivalence for (1): Since $f_k-f_i$ has negative slope and becomes 0 at $t_{ik}$, $f_k(t_{ij})>f_i(t_{ij})$ if and only if the slope of $V_{ik}$ which is $t_{ik}$ is greater than $t_{ij}$, the slope of $V_{ij}$. But this condition is equivalent to $p_k$ being above $V_{ij}$. The proof for (2) is similar.
\end{proof}

Figure \ref{fig: wire diagram for An stability condition} shows the use of colors to determine stability. $M_{ij}$ is stable if $L_i\cap L_j$ is below all positive wires (in blue) of intermediate slope (between those of $L_i,L_j$) and above all negative wires (in red) of intermediate slope. For the wire diagram in Figure \ref{Figure001} in the introduction we see that, in that example, all modules are stable except $M_{03}=I_2$.
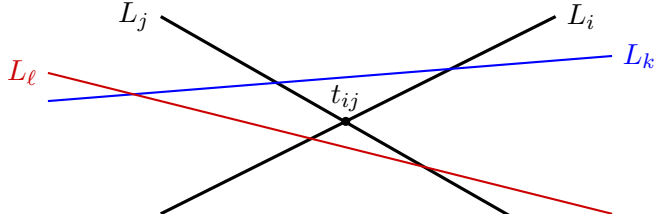
\begin{figure}[htbp]
\begin{center}
\begin{tikzpicture}[scale=.75]
\clip (-6,-1.5) rectangle (6,2.3);
\coordinate (A) at (-3.47,0.61);
\coordinate (B) at (.28,0.14);
\coordinate (C) at (-.95,0.83);
\coordinate (D) at (2.14,1.08);
\begin{scope}%[xshift=-4cm,yshift=10cm]
	\draw[very thick] (-3,-1.5)--(4,2) node[right]{$L_i$}
	(-3,2)--(4,-2) (-3,2)node[left]{$L_j$};
\end{scope}
\draw[thick,color=blue] (-5,0.5)--(5,1.3) node[right]{$L_k$};
\draw[thick,color=red!80!black] (5,-1.5)--(-5,1) node[left]{$L_\ell$};
\draw[fill] (B) circle[radius=2pt] (B) node[above]{$t_{ij}$};
\end{tikzpicture}
\caption{$M_{ij}$ is stable when all positive lines $L_k$ for $i<k<j$ (in blue) are above and all negative lines $L_\ell$ for $i<\ell<j$ (in red) are below $L_i\cap L_j$.}
\label{fig: wire diagram for An stability condition}
\end{center}
\end{figure}
%

% Section:

%\newpage
%%%%%%%%%%%%%%%%%%%%%%%%%%
%
%                Section  {3wallsD}
%
%%%%%%%%%%%%%%%%%%%%%%%%%%

\setcounter{section}2

\section{Wall crossing and maximal green sequences}\label{sec3} % sec 3

We recall the wall crossing version of stability from Bridgeland \cite{B}, \cite{B2}, Derksen-Weyman \cite{DW}, \cite{IOTW2} which define a maximal green sequence. See \cite{PartI}, \cite{BST} for details of this particular formulation given by a green path passing through a finite sequence of ``walls'' $D(M)$. In this section, $\Lambda=KQ$ is the path algebra of an arbitrary acyclic quiver $Q$.

\subsection{Semistability sets $D(M)$}\label{ss set} The relationship between linear stability conditions $Z$ and walls $D(M)$ defined below rests on the following observation.

\begin{rem}\label{rem: relation between Z and D(M)}
A $\Lambda$-module $M$ is $Z$-stable for $Z(x)=\aaa\cdot x+i\bbb\cdot x$ with slope $\sigma_Z(M)=0$ if and only if $\aaa\cdot \undim M=0$ and $\aaa\cdot \undim M'>0$ for all (nonzero) proper submodules $M'\subset M$.
\end{rem}

\begin{defn} Let $M$ be a finitely generated module over $\Lambda=KQ$.
The {\bf semistability set} $D(M)$ of $M$ is defined to be the set 
\[
D(M):=\{x\in \RR^n\,:\, x\cdot \undim M=0,\ x\cdot \undim M'\le 0\ \forall M'\subseteq M\}.
\]
The {\bf stability set} of $M$, denoted $\text{int}D(M)$, is defined to be the subset of $D(M)$ of points $x$ so that $x\cdot \undim M'<0$ for all nonzero $M'\subsetneq M$. Let $\partial D(M)=D(M)-\text{int}D(M)$. We call $\partial D(M)$ the {\bf boundary} of $D(M)$. The sets $D(M)$ are also called {\bf walls} since they divide $\RR^n$ into ``chambers'' as we explain below.
\end{defn}

Note that $D(M)$ is a convex subset of the {\bf hyperplane}
\[
H(M)=\undim M^\perp=\{x\in\RR^n\,:\, x\cdot \undim M=0\}.
\]
For example, when $M=S_i$ is simple, we have $D(S_i)=H(S_i)=\{x\in\RR^n\,:\, x_i=0\}$.

Remark \ref{rem: relation between Z and D(M)} translates into the following condition on the linear path
\[
	\lambda_Z(t):=\bbb t-\aaa
\]
corresponding to the central charge $Z(x)=\aaa\cdot x+i\bbb\cdot x$.

\begin{prop}\label{prop: relation between Z and D(M)}
The linear path $\lambda_Z$ crosses the wall $D(M)$ at time $t_0$ if and only if $M$ is $Z$-semistable with slope $t_0$. Furthermore $M$ is $Z$-stable if and only if $\lambda_Z(t_0)\in \emph{int}D(M)$.
\end{prop}

\begin{proof} Since $\bbb\cdot\undim M>0$, there is a unique $t_0\in \RR$ so that $\lambda_Z(t_0)\in H(M)$ or, equivalently, $\lambda_Z(t_0)\cdot \undim M=0$. Solving for $t_0$ we get $t_0=\sigma_Z(M)$. If $M$ is $Z$-semistable then, for any $M'\subseteq M$, we have $t_1=\sigma_Z(M')\ge \sigma_Z(M)=t_0$. Since $\bbb\cdot\undim M'>0$, we get:
\[
	0=\lambda_Z(t_1)\cdot \undim M'=(t_1-t_0)\bbb\cdot \undim M'+\lambda_Z(t_0)\cdot \undim M'\ge \lambda_Z(t_0)\cdot \undim M'.
\]
Therefore $\lambda_Z(t_0)\in D(M)$. The same calculation proves the converse. This proves the first statement. The second statement follow from this and the definition of $\text{int}D(M)$.
\end{proof}

We need the following important observation from \cite{PartI}. Recall that a module $M$ is called a {\bf brick} if every nonzero endomorphism of $M$ is an automorphism.

\begin{prop}\label{prop: D(M) for bricks M}
If the stability set $\emph{int}D(M)$ is nonempty, then $M$ is a brick.
\end{prop}

The simplest nontrivial example is the three dimensional algebra $KA_2^-$ with sign function $\vare=(0,-,0)$. This is the quiver $1\to 2$. There are three indecomposable right $KA_2^-$-modules $S_1, S_2,P_1$. The semistability sets $D(S_1), D(S_2), D(P_1)\subset \RR^2$ are shown in Figure \ref{FigureA2}.
\begin{figure}[htbp]
\begin{center}
\begin{tikzpicture}%[scale=3]
\coordinate (A) at (-2.8,0);
\coordinate (B) at (3,0);
\coordinate (B0) at (2.5,0);
\coordinate (C) at (0,-1.8);
\coordinate (D) at (0,2);
\coordinate (D0) at (0,1.7);
\coordinate (E) at (0,0);
\coordinate (F) at (1.7,-1.7);
\coordinate (F0) at (1.7,-1.5);
\draw[very thick] (A)--(B);
\draw[very thick] (C)--(D);
\draw[very thick] (E)--(F);
\draw (B0) node[above]{$D(S_2)$};
\draw (D0) node[right]{$D(S_1)$};
\draw (F0) node[right]{$D(P_1)$};
\draw[thick,color=green!70!black,->] (-.8,-1.2)..controls (1,-1) and (1,0).. (1.2,.8);
\draw[thick,color=green!70!black,->] (-1.2,-.8)..controls (-1,1) and (0,1).. (.8,1.2);
\draw[thick,color=green!70!black,->] (-.6,-.6)--(.8,.8);
\draw[color=green!70!black] (-.6,-1.2)node[below]{$\gamma_1$} (-.8,-.8) node{$\gamma_3$};
\draw[color=green!70!black] (-1.2,-.6)node[left]{$\gamma_2$};
\end{tikzpicture}
\caption{Semistability sets $D(M)$ for $KA_2^-$ where $A_2^-:1\to 2$ with three green paths $\gamma_1,\gamma_2,\gamma_3$. Since $S_2\subset P_1$, 
\[
D(P_1)=\{x\in\RR^2\,|\, x\cdot (1,1)=0\text{ , } x\cdot(0,1)\le0\}.\qquad\ 
\]
}
\label{FigureA2}
\end{center}
\end{figure}
\begin{defn}\label{def: green path}
By a {\bf green path} for $\Lambda$ we mean a smooth ($C^1$) path $\gamma:\RR\to\RR^n$ having the following properties.
\begin{enumerate}
\item $\gamma(t)$ has all coordinates negative for $t<<0$.
\item $\gamma(t)$ has all coordinates positive for $t>>0$.
\item Whenever $\gamma(t_0)\in D(M)$, the directional derivative of $\gamma$ in the direction $\undim M$ is positive, i.e.,
\[
	\undim M\cdot \frac {d\gamma}{dt}(t_0)>0.
\]
\end{enumerate}
We say that $\gamma$ passes through the wall $D(M)$ in the {\bf green direction} if (3) holds.
\end{defn}

Three examples of green paths are drawn in Figure \ref{FigureA2} with $\gamma_3$ being linear. The vector $(1,1)$ is always green. So, we see that each $\gamma_i$ passes through the walls in the green direction.

The path $\gamma(t)=(t,t,\cdots,t)$ ($\gamma_3$ in Figure \ref{FigureA2}) is always green. More generally, any linear function $\lambda(t)=\aaa+\bbb t$ for $\aaa,\bbb\in \RR^n$ is a green path if all coordinates of $b$ are positive. We call such a function a {\bf linear green path}. We say that $\lambda=\aaa+\bbb t$ is a {\bf standard linear path} if $\bbb=(1,1,\cdots,1)$. (The corresponding central charge is $Z(x)=-\aaa\cdot x+i\bbb\cdot x$.)

\begin{defn}
Given a green path $\gamma$ for $\Lambda$, a $\Lambda$-module $M$ and $t_0\in \RR$, the pair $(M,t_0)$ is called {\bf $\gamma$-stable}, resp. {\bf $\gamma$-semistable}, if $\gamma(t_0)\in \text{int}D(M)$, resp $\gamma(t_0)\in D(M)$.
\end{defn}

\subsection{Properties of green paths}\label{sec32}

One of the fundamental properties which holds for any $\Lambda$ is the following.

\begin{lem}\label{lem: uniqueness of tM}
If $\gamma$ is a green path and $(M,t_0)$ is $\gamma$-stable, then $\{t_0\}=\gamma^{-1}H(M)$. 
\end{lem}

This implies that $t_0$ is uniquely determined. We denote it $t_M$. The proof of Lemma \ref{lem: uniqueness of tM} follows the proof of Theorem \ref{thm: order thms, general} below.

\begin{defn}
A green path $\gamma$ will be called {\bf finite} if there are only finitely many modules $M_1,\cdots,M_m$ (up to isomorphism) for which $\gamma$ passes through $D(M_i)$ and if, furthermore, $\gamma$ meets the interior of each $D(M_i)$ at distinct times $t_i$ and $t_1<t_2<\cdots<t_m$. By a {\bf maximal green sequence (MGS)} for $\Lambda$ we mean a finite ordered sequence $M_1,\cdots,M_m$ given by some finite green path $\gamma$. The MGS will be called {\bf linear}, resp. {\bf standard linear}, if it is given by a linear green path, resp. standard linear path.
\end{defn}

For any green path we have: $\Hom_\Lambda(M_i,M_j)=0$ whenever $i<j$. This is Corollary \ref{cor: hom orthogonality} which follows easily from the following theorem that we will need later in this paper. (See \cite{PartI}, \cite{PartII}, \cite{BST} for more details and other equivalent definitions of a maximal green sequence.)

\begin{thm}\label{thm: order thms, general}
Suppose that $\gamma$ is a green path for $\Lambda$ and $(M,t_0)$ is $\gamma$-stable. Then $\gamma$ crosses the hyperplanes of proper submodules $M'\subsetneq M$ after $t_0$ and it crosses the hyperplanes of proper quotient modules $M''$ before $t_0$. More precisely:
\begin{enumerate}
\item[(a)] Whenever $\gamma(t')\in H(M')$ for some $M'\subsetneq M$ then $t_0<t'$.
\item[(b)] Whenever $\gamma(t'')\in H(M'')$ for some proper quotient module $M''$ of $M$ then $t''< t_0$.
 \end{enumerate}
\end{thm}

\begin{proof}
(a) For every $x$ in the hyperplane $H(M)=(\undim M)^\perp$, let $h_x:\RR\to \RR^n$ be the linear function $h_x(t)=x+t\undim M$. The image of $h_x$ is a straight line perpendicular to $H(M)$ passing through $x=h_x(0)$. Let $t=\mu(x)\in\RR$ be the smallest real number so that $h_x(t)\in H(M')$ for some proper submodule $M'$ of $M$. In other words,
\[
	x\cdot \undim M'+t\undim M\cdot\undim M'=0.
\]
Since $\undim M\cdot\undim M'\ge \dim M'>0$, we can divide to get:
\[
	\mu(x)=\min_{M'} \frac{-x\cdot \undim M'}{\undim M\cdot \undim M'}
\]
Since there are only finitely many dimension vectors of the form $\undim M'$, $\mu$ is the minimum of a finite collection of linear functions. So, $\mu$ is continuous.

%When $t_0=\mu(x)$ 
For every $x\in H(M)$ we have the following.
\begin{enumerate}
\item $h_x(\mu(x))\in H(M')$ for some $M'\subsetneq M$.
\item $h_x(t)\cdot \undim M'\le0$ for all $t\le \mu(x)$ and all $M'\subsetneq M$.
\end{enumerate}
For pairs $x,M'$ which occur in (1) it follows that:
\begin{enumerate}
\item[(3)] $h_x(\mu(x))\in D(M')$
\item[(4)] $h_x(t)\cdot \undim M'>0$ for all $t>\mu(x)$.
\end{enumerate}
Since $x=h_x(0)$ we see that
\begin{enumerate}
\item[(5)] $x$ is in the interior of $D(M)$ if and only if $\mu(x)>0$.
\end{enumerate}
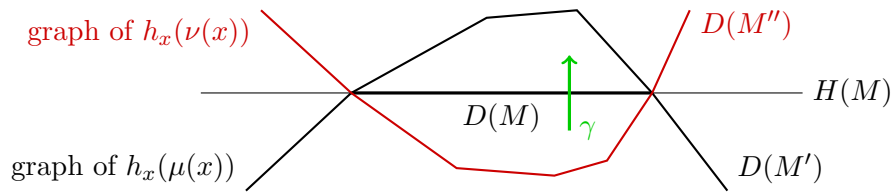
\begin{center}
\begin{tikzpicture}%[scale=3]
\coordinate (A) at (1,0);
\coordinate (L1) at (1.6,-1);
\coordinate (A1) at (1.6,-1.3);
\coordinate (B) at (3,0);
\coordinate (B1) at (4.8,1);
\coordinate (B2) at (6,1.1);
\coordinate (C) at (7,0);
\coordinate (C1) at (8,-1.3);
\coordinate (L2) at (8,-1);
\coordinate (D) at (9,0);
\coordinate (L0) at (5,0);
\coordinate (G1) at (5.9,-.5);
\coordinate (G2) at (5.9,.5);
\coordinate (nL1) at (1.8,.8);
\coordinate (nA1) at (1.8,1.1);
\coordinate (nB) at (3,0);
\coordinate (nB1) at (4.4,-1);
\coordinate (nB2) at (5.7,-1.1);
\coordinate (nB3) at (6.4,-.9);
\coordinate (nC) at (7,0);
\coordinate (nC1) at (7.5,1.1);
\coordinate (nC2) at (7.5,.9);
\draw (A)--(D);
\draw[very thick] (B)--(C);
\draw (L0) node[below]{$D(M)$};
\draw[thick] (A1)--(B)--(B1)--(B2)--(C)--(C1);
\draw[thick,color=red!80!black] (nA1)--(nB)--(nB1)--(nB2)--(nB3)--(nC)--(nC1) (nC2) node[right]{$D(M'')$};
\draw[color=red!80!black] (nL1) node[left]{graph of $h_x(\nu(x))$};
\draw (D) node[right]{$H(M)$};
\draw (L1) node[left]{graph of $h_x(\mu(x))$};
\draw (L2) node[right]{$D(M')$};
\draw[very thick,color=green!80!black,->] (G1)--(G2);
\draw[color=green!80!black] (G1) node[right]{$\gamma$};
\end{tikzpicture}
\end{center}
If $\gamma$ is a green path, it must cross the graph of the function $x\mapsto h_x(\mu(x))$ at some point. By (3), this graph is a union of semistability sets $D(M')$. So, the ``green'' condition (\ref{def: green path}(3)) means it passes from below to above. Thus, $\gamma$ can only cross this graph once from below to above. By definition of $\mu$, every hyperplane $H(M')$ for every $M'\subsetneq M$ lies on or above this graph. So, if $\gamma(t')\in H(M')$ for any such $M'$, it must be after $\gamma$ crosses this graph.

Given that $(M,t_0)$ is $\gamma$-stable, $\gamma(t_0)$ must be in the interior of $D(M)$ which, by (5), is below the graph of the function $h_x(\mu(x))$. So, $\gamma$ will hit the graph of the function after time $t_0$. This implies $t'>t_0$. So, (a) holds.

The proof of (b) is similar using $\nu(x)$, the largest real number so that $h_x(\nu(x))\in H(M'')$ for some proper quotient $M''$ of $M$.
\end{proof}

\begin{proof}[Proof of Lemma \ref{lem: uniqueness of tM}] The statement is that, for $(M,t_0)$ $\gamma$-stable, $t_0$ is unique. When $\gamma$ crosses $D(M)$, it must already have crossed the graph of $h_x(\nu(x))$ and has not yet crossed the graph of $h_x(\mu(x))$. So, it can never cross $H(M)-D(M)$ which is below the first graph and above the second. Also, $\gamma$ can cross $D(M)$ only once in the green direction.
\end{proof}

\begin{cor}\label{cor: hom orthogonality}
Let $(M_1,t_1),(M_2,t_2)$ be stable pairs for a green path $\gamma$ with $t_1<t_2$. Then $\Hom_\Lambda(M_1,M_2)=0$.
\end{cor}

\begin{proof} Suppose there is a nonzero morphism $f:M_1\to M_2$. Let $X=f(M_1)\subseteq M_2$. The path $\gamma$ must pass through the hyperplane $H(X)$ at least once since it starts from its negative side and ends on its positive side. Let $t_x$ be one of these times. Then, by (b) in Theorem \ref{thm: order thms, general}, we have $t_x< t_1$ if $X$ is a proper quotient of $M_1$ and $t_x=t_1$ by Lemma \ref{lem: uniqueness of tM} if $X=M_1$. Similarly, by (a) and Lemma \ref{lem: uniqueness of tM}, $t_2\le t_x$ since $X\subseteq M_2$. This contradicts the assumption that $t_1<t_2$.
\end{proof}

\begin{rem}
It is show in \cite{PartI} for $\Lambda$ hereditary and in \cite{PartII} for $\Lambda$ cluster-tilted of finite type (which includes all examples in this paper) that maximal green sequences are characterized by this hom-orthogonality condition. More precisely, a sequence of bricks $M_1,\cdots,M_m$ is a MGS if and only if $\Hom_\Lambda(M_i,M_j)=0$ for $i<j$ and the sequence $(M_i)$ is maximal with this property.
\end{rem}

Lemma \ref{lem: uniqueness of tM} and Theorem \ref{thm: order thms, general} give necessary conditions for $(M,t_0)$ to be $\gamma$-stable. However, we need necessary and sufficient conditions such as the following.

\begin{thm}\label{thm: linear approximation}
Let $\gamma$ be a green path and $\lambda(t)=\aaa+\bbb t$ any linear green path for $\Lambda$ so that
\[
	\lambda(t_0)=\gamma(t_0)\in H(M)
\]for some $\Lambda$-module $M$ and $t_0\in\RR$. Then $(M,t_0)$ is $\gamma$-stable if and only if, for every nonzero proper submodule $M'\subsetneq M$, there is a $t'> t_0$ so that $\lambda(t')\in H(M')$.
\begin{equation}\label{boxed eq}
	{(\forall M'\subsetneq M)(\exists t'>t_0)\ \lambda(t')\in H(M')}
\end{equation}
\end{thm}

\begin{proof} Let $Z$ be the central charge given by $Z(x)=-\aaa\cdot x+i\bbb\cdot x$. Then the condition $\lambda(t_0)\in H(M)$ is equivalent to the equation $t_0=\sigma_Z(M)$ and \eqref{boxed eq} above is equivalent to $t'=\sigma_Z(M')>t_0$. This equation for all nonzero $M'\subsetneq M$ is equivalent to the statement that $\lambda(t_0)=\gamma(t_0)\in \text{int}D(M)$. See Proposition \ref{prop: relation between Z and D(M)} and its proof for details.
\end{proof}

\begin{rem}\label{rem: dual linear approximation}
Equivalently, $(M,t_0)$ is $\gamma$-stable if and only if, for every nonzero proper quotient module $M''$ of $M$, there is a $t''<t_0$ so that $\lambda(t'')\in H(M'')$.
\end{rem}

\subsection{Affine quivers of type $\widetilde A$}\label{ss affine A}

In this paper we are interested in quivers of type $\widetilde A_{a,b}$. These have $n=a+b$ vertices labeled with integers modulo $n$ and $n$ arrows arranged in a circle with $a$ arrows going clockwise and $b$ going counterclockwise. There are $2^n$ possible orientations of such a quiver which we indicate with a sign function $\vare:[1,n]\to \{+,-\}$ or, equivalently, an $n$-periodic sign function $\vare:\ZZ\to \{+,-\}$. The sign $\vare_i$ is positive or negative depending on whether there is an arrow $i\leftarrow i+1$ or $i\to i+1$, respectively.

For example we have:
\begin{equation}\label{eq: quiver of type tilde A23}
\xymatrix{%begin xy matrix
\widetilde A_{2,3}^{-++--}:&1\ar[r]^{-} & 2 & 3\ar[l]_{+} & 4\ar[l]_{+}\ar[r]^{-} & 5 \ar@/_2pc/[llll]_{-}
	}%end xy matrix
\end{equation}
The exponent indicates the sign function. 

All quivers of type $\widetilde A_{a,b}$, $a,b\ge1$, are of infinite type. They have infinitely many indecomposable representations up to isomorphism. However, in a maximal green sequence it is well-known that only ``string modules'' occur. These are the modules $M_{ij}$, $i<j$ which we now describe.

We write the finite cyclic quiver as an infinite $n$-periodic linear quiver:
\[
\xymatrix{%begin xy matrix
\cdots\ar[r]^{\alpha_{-1}}& 0\ar[r]^{\alpha_0}& 1\ar[r]^{\alpha_1} & 2 & 3\ar[l]_{\alpha_2} & 4\ar[l]_{\alpha_3}\ar[r]^{\alpha_4} & 5 \ar[r]^{\alpha_5} & 6\ar[r]^{\alpha_6} & 7 & 8\ar[l]_{\alpha_7} & \cdots\ar[l]_{\alpha_8}
	}%end xy matrix
\]
with vertices and arrows labeled with integers so that the direction of $\alpha_i$ is the same as that of $\alpha_{i+kn}$ for any integer $k$. The direction of these arrows is given by the $n$-periodic version of the sign function 
\[
	\vare=(\cdots,-,-,-,+,+,-,-,-,+,+,\cdots)
\]
which we denote $\vare=(\vare_1,\cdots,\vare_n)$ which is $(-,+,+,-,-)$ in this case. This periodic linear quiver is the universal covering of the quiver of type $\widetilde A_{2,3}^{-++--}$ from \eqref{eq: quiver of type tilde A23}. The representation $M_{ij}$ of $\widetilde A_{a,b}^\vare$ is ``pushed-down'' from the indecomposable representation $\widetilde M_{ij}$ of the infinite quiver with dimension vector $e_{i+1}+\cdots+e_j$, just as in the finite case. The {\bf push-down} $\overline M$ of a representation $M$ of the infinite quiver is given by $\overline M_i=\bigoplus_k M_{i+kn}$ with negative arrows $i\to i+1$ inducing $\sum_k M_{\alpha_{i+kn}}:M_{i+kn}\to M_{i+kn+1}$ and similarly for positive arrows.

% Section:

%\newpage
%%%%%%%%%%%%%%%%%%%%%%%%%%
%
%                Section  {4wiresD}
%
%%%%%%%%%%%%%%%%%%%%%%%%%%

\setcounter{section}3

\section{Wire diagrams for $\widetilde A_{a,b}$}\label{sec4} 

Any green path can be written as 
\[
	\gamma(t)=t\bbb(t)-\aaa(t)
\]
where $\aaa,\bbb:\RR\to \RR^n$ are $C^1$ functions with velocity vectors $\aaa'(t),\bbb'(t)$ equal to zero for $|t|$ large (giving four vectors $\aaa(\infty),\aaa(-\infty),\bbb(\infty),\bbb(-\infty)$) and so that $b_i(t)>0$ for all $t$ and $i$. For example, we could let $\bbb(t)$ be the constant vector $\bbb=(1,1,\cdots,1)$ and let $\aaa(t)=t\bbb-\gamma(t)$. Although the decomposition of $\gamma$ into the two parts $t\bbb(t)$ and $-\aaa(t)$ is not unique, it gives a very useful interpretation of $\gamma$ for quivers of type $A_n$ and $\widetilde A_{a,b}$. Also, in Theorem \ref{thm: linear approximation} we can take the linear green path $\lambda$ to be $\lambda_{t_0}$ given by
\begin{equation}\label{eq: lambda t0}
	\lambda_{t_0}(t):=t\bbb(t_0)-\aaa(t_0)
\end{equation}
since, clearly, $\lambda_{t_0}(t_0)=\gamma(t_0)$.

\subsection{Slope} Given a fixed decomposition $\gamma(t)=t\bbb(t)-\aaa(t)$, we can define the {\bf slope} of any nonzero $\Lambda$-module $M$ at time $t$ by
\[
	\sigma_t(M):=\frac{\aaa(t)\cdot \undim M}{\bbb(t)\cdot \undim M}.
\]
We have the following easy calculation.
\begin{equation}\label{eq: when g(t) is in H(M)}
\sigma_t(M)=t \ifff \gamma(t)\cdot \undim M=0\ifff \gamma(t)\in H(M).
\end{equation}

The following characterization of stable pairs follows from Theorem \ref{thm: linear approximation} using $\lambda=\lambda_{t_0}$.

\begin{thm}\label{thm: characterization of stability using slope} Let $\gamma$ be a green path with slope function $\sigma_t$. Then $(M,t_0)$ is a $\gamma$-stable pair, i.e., $\gamma(t_0)\in $\emph{ int}$D(M)$, if and only if the following conditions are satisfied.
\begin{enumerate}
\item $t_0=\sigma_{t_0}(M)$
\item $\sigma_{t_0}(M')> t_0$, for all nonzero $M'\subsetneq M$.
\end{enumerate}
\end{thm}
Note that, since the numerator of $\sigma_{t_0}(M)-\sigma_{t_0}(M')$ is equal to the numerator of $\sigma_{t_0}(M/M')-\sigma_{t_0}(M)$, condition (2) is equivalent to the following condition:\vs2

(3) \emph{$\sigma_{t_0}(M'')<t_0$ for all proper quotient modules $M''$ of $M$.
}\vs2

\begin{proof}
By \eqref{eq: when g(t) is in H(M)}, (1) is equivalent to the condition $\gamma(t_0)\in H(M)$. Now apply Theorem \ref{thm: linear approximation} using $\lambda=\lambda_{t_0}$ defined in \eqref{eq: lambda t0}. For any $M'\subsetneq M$, the formula for $\sigma_{t_0}$ shows that $t'=\sigma_{t_0}(M')$ is the unique real number so that $\lambda_{t_0}(t')\in H(M')$. Therefore, (2) is equivalent to the condition $t'>t_0$ in Theorem \ref{thm: linear approximation} which is equivalent to $(M,t_0)$ being $\gamma$-stable.
\end{proof}

Next, we use the well-known fact that, for $\Lambda$ hereditary, the modules $M$ in any maximal green sequence are {\bf exceptional} which means that $\End_\Lambda(M)=K$ and $\Ext_\Lambda^1(M,M)=0$. We also use the well-known fact that $\Lambda$ of type $\widetilde A_{a,b}$ is a ``gentle algebra'' in which all exceptional modules are ``string modules'' of the form $M_{ij}$ for $i<j$ with $j-i$ not divisible by $n$ with dimension vector
\[
	\undim M_{ij}=\beta_{ij}=\sum_{k=i+1}^j e_k
\]
where $e_k$ is the $p$-th unit vector for $p\in[1,n]$ congruent to $k$ modulo $n$. The indices $i,j$ satisfy one more condition: when $\vare_i=\vare_j$, we must have $|i-j|<n$, otherwise $M_{ij}$ is not exceptional. This condition also appears later in Proposition \ref{prop: regular stable modules}.

\subsection{The functions $f_i(t)$} For $\Lambda$ of type $\widetilde A_{a,b}$, the ``wires'' in our ``wire diagram'' are the graphs of functions $f_i:\RR\to\RR$, $i\in\ZZ$, given by
\begin{equation}\label{functions fi}
	f_i(t)=m_i(t)t+c_i(t)
\end{equation}
where $c_i=a_1+a_2+\cdots+a_i$, taking indices modulo $n$, and $m_i=m_0-b_1-b_2-\cdots-b_i$ for some fixed smooth function $m_0:\RR\to\RR$.

\begin{lem}\label{lem: Tij=gamma-1 H(Mij)}
For $i<j$ and $t_0\in \RR$, the following are equivalent.
\begin{enumerate}
\item $f_i(t_0)=f_j(t_0)$. 
\item $\gamma(t_0)\in H(M_{ij})$. 
\end{enumerate} 
\end{lem}

\begin{proof} Since $f_i(t)-f_j(t)=(b_{i+1}+\cdots+b_j)t-(a_{i+1}+\cdots+a_j)$ the values of $t_0$ so that $f_i(t_0)=f_j(t_0)$ are given by
\[
	t_0=\frac{a_{i+1}+\cdots+a_j}{b_{i+1}+\cdots+b_j}=\frac{\aaa(t)\cdot \undim M_{ij}}{\bbb(t)\cdot \undim M_{ij}}=\sigma_{t_0}(M_{ij})
\]
So, (1) is equivalent to $t_0=\sigma_{t_0}(M_{ij})$ which, by \eqref{eq: when g(t) is in H(M)}, is equivalent to (2).
\end{proof}

Let $T_{ij}=\{t_{ij}\}$ denote the set of all real numbers so that $f_i(t_{ij})=f_j(t_{ij})$. Thus $T_{ij}=T_{ji}=\gamma^{-1}H(M_{ij})$ by Lemma \ref{lem: Tij=gamma-1 H(Mij)} above. Since $f_i(t)<f_j(t)$ for $t<<0$ and $f_i(t)>f_j(t)$ for $t>>0$ (when $i<j$), $T_{ij}$ is always nonempty. The smallest and largest elements of $T_{ij}$ will be denoted $t_{ij}^0$ and $t_{ij}^1$ respectively.

The order theorems, from Section \ref{sec32}, which hold for any finite dimensional algebra imply the following for $\widetilde A_{a,b}$.

\begin{lem}\label{lem: tij stable}
Let $M_{ij}$ be $\gamma$-stable for a green path $\gamma$. Then $t_{ij}\in T_{ij}$ is unique.
\end{lem}

\begin{proof}
By Lemma \ref{lem: uniqueness of tM}, $\gamma^{-1}H(M_{ij})=T_{ij}$ has only one element.
\end{proof}

\begin{thm}\label{thm: stable iff blue above and red below} Suppose that $t_{ij}$ is unique. Then $(M_{ij},t_{ij})$ is $\gamma$-stable if and only if
\begin{enumerate}
\item for all $i<k<j$ with $\varepsilon_k=+$ we have $f_k(t_{ij})>f_i(t_{ij})$ and 
\item for all $i<\ell<j$ with $\varepsilon_\ell=-$ then $f_\ell(t_{ij})<f_i(t_{ij})$.
\end{enumerate}
\end{thm}

\begin{rem}\label{rem: red and blue lines}
Paraphrased: $M_{ij}$ is stable if and only if $L_i$, $L_j$ meet at only one point $t_{ij}$ and all positive (blue) curves $L_k$ with $i<k<j$ go over that point and all negative (red) curves $L_\ell$ with $i<\ell<j$ go under that point.
\end{rem}

\begin{proof} We again use Theorem \ref{thm: linear approximation} with $\lambda(t)=t\bbb(t_{ij})-\aaa(t_{ij})$. Let $L_i^\lambda,L_j^\lambda$, etc. denote the linear wires for $\lambda$. Let $t_{ij}^\lambda$ be the $x$-coordinate of $L^\lambda_i\cap L^\lambda_j$. If $L^\lambda_k$ lies above $L^\lambda_i\cap L^\lambda_j$ and $L^\lambda_\ell$ lies below $L^\lambda_i\cap L^\lambda_j$ then the $t$ coordinates of their intersection points are arranged:
\[
	t_{ij}=t^\lambda_{ij}<t^\lambda_{ik},t^\lambda_{\ell j}
\]Also, $t^\lambda_{ij}<t^\lambda_{\ell k}$ when $\ell<k$.
Since $\varepsilon_k=+$ and $\vare_\ell=-$, $M_{ik}\subsetneq M_{ij}$, $M_{\ell j}\subsetneq M_{ij}$ and $M_{\ell k}\subsetneq M_{ij}$ when $\ell<k$. These are all of the proper indecomposable submodules of $M_{ij}$. So, by Theorem \ref{thm: linear approximation}, $M_{ij}$ is $\gamma$-stable. 
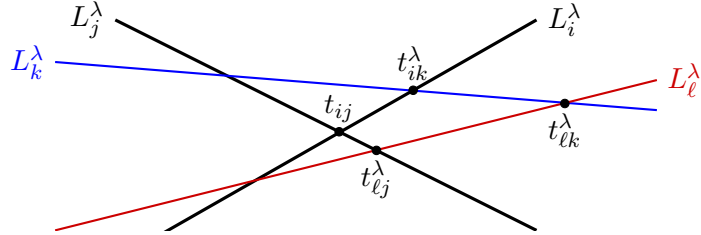
\begin{center}
\begin{tikzpicture}[scale=.8]
\clip (-6,-1.5) rectangle (6,2.3);
\coordinate (A) at (3.47,0.61);
\coordinate (B) at (-.28,0.14);
\coordinate (C) at (.95,0.83);
\coordinate (D) at (.34,-.17);
\begin{scope}%[xshift=-4cm,yshift=10cm]
	\draw[very thick] (3,-1.5)--(-4,2) node[left]{$L^\lambda_j$}
	(3,2)--(-4,-2) (3,2)node[right]{$L^\lambda_i$};
\end{scope}
\draw[thick,color=blue] (5,0.5)--(-5,1.3) node[left]{$L^\lambda_k$};
\draw[thick,color=red!80!black] (-5,-1.5)--(5,1) node[right]{$L^\lambda_\ell$};
\draw[fill] (A) circle[radius=2pt] (A) node[below]{$t^\lambda_{\ell k}$};
\draw[fill] (B) circle[radius=2pt] (B) node[above]{$t_{ij}$};
\draw[fill] (C) circle[radius=2pt] (C) node[above]{$t^\lambda_{ik}$};
\draw[fill] (D) circle[radius=2pt] (D) node[below]{$t^\lambda_{\ell j}$};
\end{tikzpicture}
\end{center}
Conversely, if $M_{ij}$ is $\gamma$-stable then, by Theorem \ref{thm: linear approximation}, $t_{ij}$ must be less than $t^\lambda_{ik}$ and $t^\lambda_{\ell j}$ for $i<k,\ell<j$, $\vare_k=+$, $\vare_\ell=-$. This implies that the straight line $L^\lambda_k$ goes above $L^\lambda_i\cap L^\lambda_j$ and $L^\lambda_\ell$ goes below $L^\lambda_i\cap L^\lambda_j$. But $\gamma(t_{ij})=\lambda(t_{ij})$. So, $f_\ell(t_{ij})<f_i(t_{ij})<f_k(t_{ij})$ as claimed.
\end{proof}

When $\gamma$ is linear, we can reinterpret Theorem \ref{thm: stable iff blue above and red below} in terms of {\bf chords} $V_{ij}$ as follows.

\begin{cor}\label{cor: stability in terms of chords Vij}
Let $\gamma(t)=\bbb t-\aaa$ be the linear green path with $\aaa,\bbb\in\RR^n$ constant with $b_i>0$ for every $i$, i.e., $\gamma=\lambda_Z$ for the central charge $Z$ given by $Z(x)=\aaa\cdot x+i\bbb\cdot x$. Let $p_i\in \RR^2$ be given by 
\[
	p_i=(b_1+b_2+\cdots+b_i,a_1+a_2+\cdots+a_i)
\]
Let $V_{ij}$ be the line segment from $p_i$ to $p_j$.
Then $M_{ij}$ is $\gamma$-stable if and only if the following two conditions are satisfied:
\begin{enumerate}
\item For every $i<k<j$ with $\varepsilon_k=+$ the point $p_k$ lies above the line segment $V_{ij}$.
\item For every $i<\ell<j$ with $\varepsilon_\ell=-$ the point $p_\ell$ lies below $V_{ij}$.
\end{enumerate}
\end{cor}

\begin{proof}
The slope of $V_{ij}$, which is equal to $\sigma_Z(M_{ij})=t_{ij}$, is less than $t_{ik}$ and $t_{\ell j}$, the slopes of $V_{ik}$ and $V_{\ell j}$, if and only if Conditions (1) and (2) hold. \end{proof}

\begin{thm} Let $M_{ij}$ be $\gamma$-stable for a green path $\gamma$. Then $t_{ij}$ is unique and we have the following.

\emph{(1)} Suppose $i<k<j$ and $\vare_k=+$. Then
\[
	t^1_{kj}<t_{ij}<t^0_{ik}.
\]

\emph{(2)} Similarly, for $i<\ell<j$, and $\vare_\ell=-$ we have
\[
	t^1_{i\ell}<t_{ij}<t^0_{\ell j}.
\]
Conversely, when $(1)$, $(2)$ both hold and $t_{ij}$ is unique we have:
\[
f_\ell(t_{ij})<f_i(t_{ij})=f_j(t_{ij})<f_k(t_{ij}).
\]
\end{thm}

These three statements are illustrated in Figure \ref{Fig: OrderThms012}.
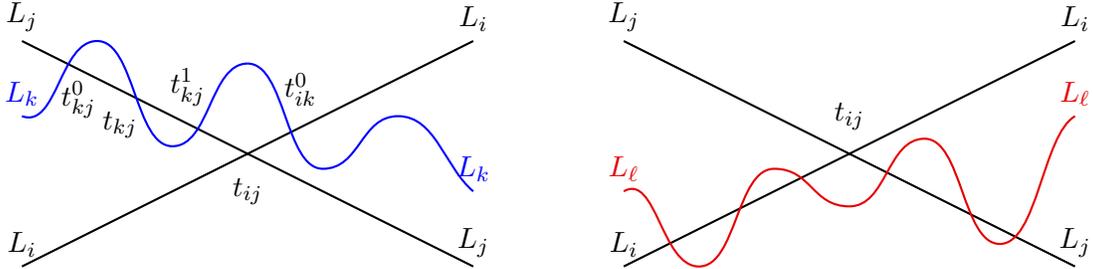
\begin{figure}[htbp]%\label{Fig: OrderThms012}
\begin{center}
\begin{tikzpicture}%[scale=3]
\coordinate (i0) at (0,0);
\coordinate (k0) at (0,2);
\coordinate (j0) at (0,3);
\coordinate (k01) at (1,3);
\coordinate (k02) at (2,1.6);
\coordinate (k03) at (3,2.7);
\coordinate (k04) at (4,1.3);
\coordinate (k05) at (5,2);
\coordinate (j1) at (6,0);
\coordinate (k1) at (6,1);
\coordinate (i1) at (6,3);
\draw (i0) node[above]{$L_i$};
\draw (j0) node[above]{$L_j$};
\draw[color=blue] (k0) node[above]{$L_k$};
\draw (i1) node[above]{$L_i$};
\draw (j1) node[above]{$L_j$};
\draw[color=blue] (k1) node[above]{$L_k$};
\draw[thick] (i0)--(i1) (j0)--(j1);
\draw (3,1)node{$t_{ij}$};
\draw (3.7,2)node[above]{$t^0_{ik}$};
\draw (.75,2.6)node[below]{$t^0_{kj}$};
\draw (1.3,2.2)node[below]{$t_{kj}$};
\draw (2.2,2)node[above]{$t^1_{kj}$};
\draw[color=blue,thick] (k0)..controls +(-20:.5cm) and +(180:.5cm).. (k01).. controls +(0:.5cm) and +(180:.5cm).. (k02) .. controls +(0:.5cm) and +(180:.5cm)..(k03);
\draw[color=blue,thick] (k03)..controls +(0:.5cm) and +(180:.5cm).. (k04).. controls +(0:.5cm) and +(180:.5cm).. (k05) .. controls +(0:.5cm) and +(140:.5cm)..(k1);
 %
 % righthand graph
\coordinate (i2) at (8,0);
\coordinate (k2) at (8,1);
\coordinate (j2) at (8,3);
\coordinate (k11) at (9,0);
\coordinate (k12) at (10,1.3);
\coordinate (k13) at (11,.8);
\coordinate (k14) at (12,1.7);
\coordinate (k15) at (13,.3);
\coordinate (j3) at (14,0);
\coordinate (k3) at (14,2);
\coordinate (i3) at (14,3);
\draw (i2) node[above]{$L_i$};
\draw (j2) node[above]{$L_j$};
\draw[color=red!90!black] (k2) node[above]{$L_\ell$};
\draw (i3) node[above]{$L_i$};
\draw (j3) node[above]{$L_j$};
\draw[color=red!90!black] (k3) node[above]{$L_\ell$};
\draw[thick] (i2)--(i3) (j2)--(j3);
\draw (11,2)node{$t_{ij}$};
\draw[color=red!90!black,thick] (k2)..controls +(30:.5cm) and +(180:.5cm).. (k11).. controls +(0:.5cm) and +(180:.5cm).. (k12) .. controls +(0:.5cm) and +(180:.5cm)..(k13);
\draw[color=red!90!black,thick] (k13)..controls +(0:.5cm) and +(180:.5cm).. (k14).. controls +(0:.5cm) and +(180:.5cm).. (k15) .. controls +(0:.5cm) and +(210:.5cm)..(k3);
\end{tikzpicture}
\caption{The wires $L_i$, $L_j$ cross only once, at $t_{ij}$, since $M_{ij}$ is stable. The statement $t_{ij}<t^0_{ik}$ (the first crossing of $L_i,L_k$) is equivalent to saying that $L_k$ does not cross $L_i$ to the left of $t_{ij}$. Similarly $t^1_{kj}<t_{ij}$ says $L_k$ does not cross $L_j$ to the right of $t_{ij}$. (All three values of $t_{kj}$ are to the left of $t_{ij}$.)}
\label{Fig: OrderThms012}
\end{center}
\end{figure}

\begin{proof}
As in the case of finite $A_n$ quivers, the positivity of $k$ implies that the $k$th arrow points to the left: $M_{ik}\leftarrow \cdots$. Therefore, $M_{ik}\subsetneq M_{ij}$. By Theorem \ref{thm: order thms, general}, $t_{ij}<t'$ for any $t'\in \gamma^{-1}H(M_{ik})$. We are using the notation $t'=t_{ik}$. So, $t_{ij}<t_{ik}$. The other three cases are similar. The converse statement follows by examination of Figure \ref{Fig: OrderThms012}. E.g., $L_k$ cannot go under $L_i\cap L_j$ without crossing $L_i$ to the left of $t_{ij}$.
\end{proof}

\begin{defn}\label{def: witness}
If $i<k<j$ with $\vare_k=+$ and either $t_{ik}^0<t_{ij}$ or $t_{ij}<t_{kj}^1$ we say that $k$ is a {\bf witness} to the instability of $M_{ij}$. Similarly, $\ell$ is a {\bf witness} to the instability of $M_{ij}$ if $i<\ell<j$ with $\vare_\ell=-$ satisfies either $t_{\ell j}^0<t_{ij}$ or $t_{ij}<t_{i\ell}^1$.
\end{defn}

\begin{thm}\label{thm: 4lines}
Let $M_{ij}$ be $\gamma$-stable for a green path $\gamma$. Then $t_{ij}$ is unique and we have the following.

\emph{(3)} Whenever $i<k<\ell<j$ with $\vare_k=+,\vare_\ell=-$ we have:
\[
	t^1_{k\ell}<t_{ij}.
\]

\emph{(4)} Whenever $i<\ell<k<j$ with $\vare_k=+,\vare_\ell=-$ we have:
\[
	t_{ij}<t^0_{\ell k}.
\]
\end{thm}

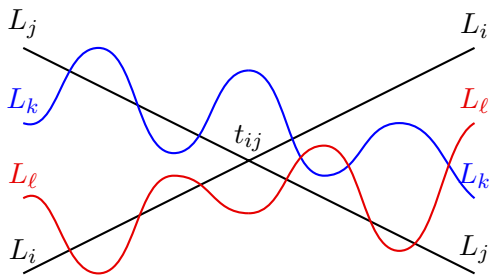
\begin{figure}[htbp]%\label{Fig: OrderThms34}
\begin{center}
\begin{tikzpicture}%[scale=3]
\coordinate (i0) at (0,0);
\coordinate (k0) at (0,2);
\coordinate (j0) at (0,3);
\coordinate (k01) at (1,3);
\coordinate (k02) at (2,1.6);
\coordinate (k03) at (3,2.7);
\coordinate (k04) at (4,1.3);
\coordinate (k05) at (5,2);
\coordinate (j1) at (6,0);
\coordinate (k1) at (6,1);
\coordinate (i1) at (6,3);
\draw (i0) node[above]{$L_i$};
\draw (j0) node[above]{$L_j$};
\draw[color=blue] (k0) node[above]{$L_k$};
\draw (i1) node[above]{$L_i$};
\draw (j1) node[above]{$L_j$};
\draw[color=blue] (k1) node[above]{$L_k$};
\draw[thick] (i0)--(i1) (j0)--(j1);
\draw (3,1.8)node{$t_{ij}$};
\draw[color=blue,thick] (k0)..controls +(-20:.5cm) and +(180:.5cm).. (k01).. controls +(0:.5cm) and +(180:.5cm).. (k02) .. controls +(0:.5cm) and +(180:.5cm)..(k03);
\draw[color=blue,thick] (k03)..controls +(0:.5cm) and +(180:.5cm).. (k04).. controls +(0:.5cm) and +(180:.5cm).. (k05) .. controls +(0:.5cm) and +(140:.5cm)..(k1);
 %
 % righthand graph
 \begin{scope}[xshift=-8cm]
\coordinate (i2) at (8,0);
\coordinate (k2) at (8,1);
\coordinate (j2) at (8,3);
\coordinate (k11) at (9,0);
\coordinate (k12) at (10,1.3);
\coordinate (k13) at (11,.8);
\coordinate (k14) at (12,1.7);
\coordinate (k15) at (13,.3);
\coordinate (j3) at (14,0);
\coordinate (k3) at (14,2);
\coordinate (i3) at (14,3);
\draw[color=red!90!black] (k2) node[above]{$L_\ell$};
\draw[color=red!90!black] (k3) node[above]{$L_\ell$};
\draw[color=red!90!black,thick] (k2)..controls +(30:.5cm) and +(180:.5cm).. (k11).. controls +(0:.5cm) and +(180:.5cm).. (k12) .. controls +(0:.5cm) and +(180:.5cm)..(k13);
\draw[color=red!90!black,thick] (k13)..controls +(0:.5cm) and +(180:.5cm).. (k14).. controls +(0:.5cm) and +(180:.5cm).. (k15) .. controls +(0:.5cm) and +(210:.5cm)..(k3);
\end{scope}
\end{tikzpicture}
\caption{Case (4): if $M_{ij}$ is stable and $i<\ell<k<j$ with $\vare_k=+,\vare_\ell=-$, the graphs of $f_k$ $f_\ell$ do not cross on the left side of $t_{ij}$. (In this example, all three values of $t_{\ell k}$ are $>t_{ij}$.)}
\label{Fig: OrderThms34}
\end{center}
\end{figure} % end {Fig: OrderThms34}

\begin{proof}
In Case (4), when $\ell<k$ and $\vare_k=+,\vare_\ell=-$, the arrow pattern is: $\rightarrow M_{\ell k}\leftarrow$ making $M_{\ell k}$ a proper submodule of $M_{ij}$. So, $t'>t_{ij}$ for all $t'=t_{\ell k}\in T_{\ell k}$ by Theorem \ref{thm: order thms, general}. In Case (3), $M_{k\ell}$ is a proper quotient module of $M_{ij}$. So, $t''<t_{ij}$ for all $t''=t_{k\ell}\in T_{k\ell}$.% in Case (3).
\end{proof}

\subsection{Example: a nonlinear maximal green sequence}

The wire diagram in Figure \ref{Fig: Pappus} shows an example of a nonlinear maximal green sequence for $A_5^{-+-+}$.

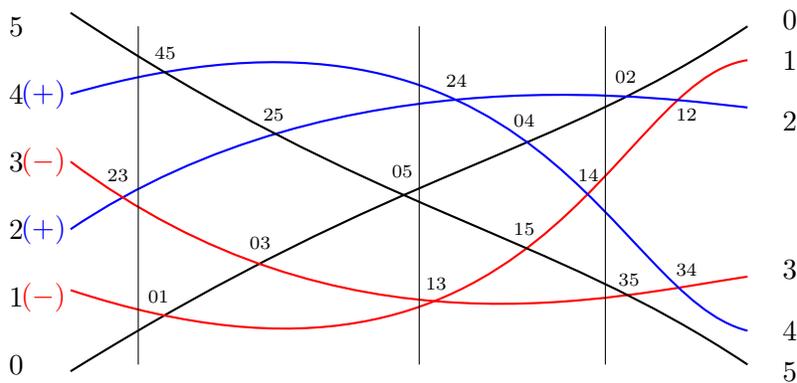
\begin{figure}[htbp]%\label{Fig: Pappus} put label after caption
\begin{center}
\begin{tikzpicture}[scale=.9] \foreach \y in {0,...,5} \draw (-5.8,\y) node{\y};
\begin{scope}[xshift=-5]
\draw (5.8, 5.1) node{0};
\draw (5.8, 4.5) node{1};
\draw (5.8, 3.6) node{2};
\draw (5.8, 1.4) node{3};
\draw (5.8, .5) node{4};
\draw (5.8, -.1) node{5};
\end{scope}
\draw[thick] (-5,-.1)..controls (0,3) and (2,3)..(5,5); % Line 0  controls (-2,1) and (-1,4)..
\draw[thick, color=red](-5.4,1) node{$(-)$} (-5,1.1) .. controls (2,-1.2) and (3,4.2)..  (5,4.5); % Line 1
\draw[thick, color=blue](-5.4,2) node{$(+)$} (-5,2) .. controls (-1,4.7) and (4,3.9)..  (5,3.8); % Line 2
\draw[thick, color=red](-5.4,3) node{$(-)$} (-5,3) .. controls  (-1,0) and (3,1)..  (5,1.3); % Line 3
\draw[thick, color=blue](-5.4,4) node{$(+)$} (-5,4) .. controls (2,6) and (3,1)..  (5,.5); % Line 4
\draw[thick] (-5,5.2)..controls (-.5,2.2) and (2,2).. (5,0); 
\draw (-4,0)--(-4,5); % first vertical line
\draw (0.15,0)--(.15,5); % second vertical line
\draw (2.9,0)--(2.9,5); % third vertical line
\draw (-3.7,1) node{\tiny$01$};
\draw (-3.6,4.6) node{\tiny$45$};
\draw (-4.3,2.8) node{\tiny$23$};
\draw (-2.2,1.8) node{\tiny$03$};
\draw (-2,3.7) node{\tiny$25$};
\draw (-.1,2.85) node{\tiny$05$};
\draw (0.7,4.2) node{\tiny$24$};
\draw (0.4,1.2) node{\tiny$13$};
\draw (1.7,2) node{\tiny$15$};
\draw (1.7,3.6) node{\tiny$04$};
\draw (3.25,1.25) node{\tiny$35$};
\draw (3.2,4.25) node{\tiny$02$};
\draw (2.65,2.8) node{\tiny$14$};
\draw (4.1,1.4) node{\tiny$34$};
\draw (4.1,3.7) node{\tiny$12$};
\end{tikzpicture}
\caption{Example of a nonlinear MGS on $A_5^{-+-+}:1\to 2\leftarrow 3\to 4\leftarrow 5$. Theorem \ref{thm: stable iff blue above and red below} shows that all modules $M_{ij}$ are stable. Pappus' Theorem implies that this arrangement of lines is nonlinear.
}
\label{Fig: Pappus}
\end{center}
\end{figure}

\begin{enumerate}
\item This is a wire diagram since each curve is the graph of a continuous function $\RR\to\RR$, all crossing are transverse and the wires $L_0,\cdots,L_5$ are in ascending order on the left and descending order on the right.
\item The diagram gives a maximal green sequence since all stable crossings are green (with lower numbered wire having higher slope).
\item All modules $M_{ij}$ are stable since $L_i,L_j$ cross only once and the intermediate blue, resp. red, wires pass above, resp. below, the crossing point. E.g., $M_{05}$ is stable since $L_2,L_4$ go over the crossing point $L_0\cap L_5$ and $L_1,L_3$ go under that point.
\item The diagram is nonlinear since, if it were linear, the wires $L_0$, $L_5$ could be moved to the left to make the crossings $t_{01},t_{45}$, resp. $t_{02},t_{35}$, line up with $t_{23}$, resp. $t_{14}$. But this would move the crossing $t_{05}$ to the left contradicting Pappus' Theorem.
\end{enumerate}

This proves the following.

\begin{thm}\label{thm: Pappus Thm example}
For the quiver $A_5^{-+-+}$, Figure \ref{Fig: Pappus} gives an example of a maximal green sequence of maximal length, 15, which is not given by any linear stability condition.\qed
\end{thm}

% Section:

%\newpage
%%%%%%%%%%%%%%%%%%%%%%%%%%
%
%                Section  {5criticalD}
%
%%%%%%%%%%%%%%%%%%%%%%%%%%

\setcounter{section}4

\section{Critical line and upper bound}\label{sec5}\label{critical line}

In this section we introduce the critical line and use the order of the points on the critical line to determine which wire crossings can be stable. We then use this to find the upper bound on the length of a maximal green sequence by assuming that all of these crossings are indeed stable. In Section \ref{sec6} we will show that this upper bound is actually attained.

Throughout this section we assume that $\gamma$ is a generic green path for a quiver of type $\widetilde A_{a,b}^\vare$ with associated functions $f_i$ and wires $L_i$ which are the graphs of these functions. Note that, since $\gamma$ is generic, no three wires meet at any point. So, $M_{ij}$ is $(\gamma)$-stable if and only if it is $(\gamma)$-semistable. We will also say that the intersection $t_{ij}$ is stable in that case.

\subsection{The critical line}\label{ss critical}
	Due to the periodicity of the wire diagram ($f_{i+n}(t)-f_i(t)=f_{j+n}(t)-f_j(t)$ for all $i,j$), we can observe that $t^0_{0,n} = t^0_{1, n+1} = t^0_{2, n+2} = \dots$, or more generally that $t^0_{i, i+n} = t^0_{j, j+n}$ for all $i, j$. This leads to the following definition.
	
	\begin{defn} Let $c^0=t^0_{0,n}$. The {\bf critical line} is the line $t = c^0$, which is the left-most vertical line upon which all the intersections $L_i\cap L_{i+n}$ fall. We say $i \triangleleft j$ when $f_i (c^0) < f_j (c^0)$ and $i\trianglelefteq j$ when $f_i (c^0) \le f_j (c^0)$.
	\end{defn}
	
	The critical line is so called because we can use it to determine several conditions for stability of intersections, to be described shortly. Note that the signs of the lines in the wire diagram determine a sequence on the critical line, such as the one in Figure \ref{fig: critical line}.
	
	\begin{figure}[htbp]
		\begin{center}
			\begin{tabular}{c c}
				$j$ & $\cdot$ \\
				$\ell$ & $-$ \\
				$k$ & $+$ \\
				$i$ & $+$
			\end{tabular}
		\end{center}	
		\caption{Four points on the critical line with $\vare_i=\vare_k=+$ and $\vare_\ell=-$. The figure indicates $i\triangleleft k\triangleleft \ell \triangleleft j$ but does not preclude the presence of other points on the critical line above, below or between these points.}\label{fig: critical line}
	\end{figure}

\subsection{Lemmas}
	For each $i, j$, we will determine all $p$ so that $M_{i, j+pn}$ can be stable. More precisely we show that $M_{i,j+pn}$ is {\bf unstable} (not semistable) for all but two values of $p$. To show that $M_{ij}$ is unstable, we find $t_{ij}\in T_{ij}$ so that $\gamma(t_{ij})\notin D(M_{ij})$. In that case we say that the intersection of the wires $L_i$ and $L_{j}$ at $t_{ij}$ is {unstable}. This will depend on the sign of $i, j$ and/or the position of other curves $L_k, L_{\ell}$ on the critical line.

\begin{lem}\label{lem for 2,3 point lemmas}
Suppose $i < k < j$ and $t_0\in\RR$ so that $f_i(t_0)\neq f_j(t_0)$. Suppose either
\begin{enumerate}
\item $\varepsilon_k = +$ and $f_k(t_0) \le \min( f_i(t_0), f_j(t_0))$ or
\item $\varepsilon_k = -$ and $f_k(t_0) \ge \max( f_i(t_0), f_j(t_0))$.
\end{enumerate}
Then $M_{ij}$ is unstable. 
\end{lem}

\begin{proof} In both cases there are two subcases: (a) $i\triangleleft j$ and (b) $j\triangleleft i$. We may assume $t_{ij}$ is unique since, otherwise, $M_{ij}$ is not stable. In Subcase (a), $t_{ij}>t_0$ and, in Subcase (b), $t_{ij}<t_0$. In all four subcases, $k$ is a witness to the instability of $M_{ij}$ (Definition \ref{def: witness}). For example, in Case (1a), $f_i(t)=f_k(t)$ for some $t\le t_0$. So, $t^0_{ik}\le t_0<t_{ij}$ making $k$ a witness. In Case (1b), $f_j(t),f_k(t)$ switch order at some $t\ge t_0$. So  $t^1_{kj}\ge t_0>t_{ij}$ making $k$ a witness. Cases (2a), (2b) are similar. Thus $M_{ij}$ is not stable.
\end{proof}
	
	\begin{lem}[Three-Point Lemma]\label{3pt lem}
		Suppose there are three points on the critical line in one of the following two arrangements:
		\begin{center}
			\begin{tabular}{c c ccc c c}
				$j$ & $\cdot$ &&&& $k$ & $-$\\
				$i$ & $\cdot$ &&\text{or}&& $j$ & $\cdot$\\
				$k$ & $+$ &&&& $i$ & $\cdot$
			\end{tabular}
		\end{center}
I.e., $k\triangleleft i\triangleleft j$ with $\vare_k=+$ in the first case and $i\triangleleft j\triangleleft k$ with $\vare_k=-$ in the second case. 
Then there is at most one stable intersection of the form $t_{i,j'}$ where $k<i<k+n$ and $j= j+pn$ for some $p$, namely the one given by $k<i,j'<k+n$.
	\end{lem}

\begin{proof}
Case 1: If $j'<k$ then $j'<k<i$ making $M_{j'i}$ unstable by Lemma \ref{lem for 2,3 point lemmas} (1). Similarly, if $k+n<j'$, then $i<k+n<j'$ making $M_{ij'}$ unstable. So, there is only one possible value of $j'=j+pn$ for which $M_{ij'}$ could be stable, namely the one with $k<j'<k+n$.

Case 2 is similar using Lemma \ref{lem for 2,3 point lemmas} (2).
\end{proof}

	\begin{lem}[Two-Point Lemma]\label{2pt lem}
		Suppose there are two points on the critical line in one of the following two arrangements:
		\begin{center}
			\begin{tabular}{c c ccc c c}
				$j$ & $\cdot$ &&&& $j$ & $-$\\
				$i$ & $+$ &&\text{or}&& $i$ & $\cdot$\\
			\end{tabular}
		\end{center}
I.e., $i\triangleleft j$ and either $\vare_i=+$ or $\vare_j=-$. Then there are at most two stable intersections of the form $t_{i, j'}$ with $j'=j+pn$ given by taking the two values of $j'$ so that $i-n<j'<i+n$.
	\end{lem}
	
\begin{proof} Take the first case.  If $j'<i-n$ then apply Lemma \ref{lem for 2,3 point lemmas} with $k=i-n$ to see that $M_{ij'}$ is unstable. If $j'>i+n$ we let $k=i+n$. Then, again, $M_{ij'}$ is unstable by Lemma \ref{lem for 2,3 point lemmas}. So, the only possible stable $M_{ij'}$ are the two with $i-n<j'<i+n$. 

The second case is similar.
\end{proof}

	\begin{lem}[Four-Point Lemma]\label{4pt lem}
		Suppose there are four points on the critical line in the following arrangement:
		\begin{center}
			\begin{tabular}{c c}
				$j$ & $\cdot$ \\
				$\ell$ & $-$ \\
				$k$ & $+$ \\
				$i$ & $\cdot$
			\end{tabular}
		\end{center}
I.e, $i\triangleleft k\triangleleft \ell\triangleleft j$,
$\vare_k=+$ and $\vare_\ell=-$. Then there are at most two stable intersections of the form $t_{i, j+pn}$. Choosing $i,j,k,\ell$ so that $k<j,\ell<k+n$ and $\ell<i<\ell+n$, the possible stable intersections are $t_{ij}$ and $t_{i, j+n}$.
	\end{lem}

\begin{proof}
This follows from Theorem \ref{thm: 4lines} using the given ordering:
\[
	k<\ell<i<\ell+n<k+2n.
\]
For any $j'=j+pn<k$ we have $t_{i,j'}<c^0<t_{k\ell}^1$. So, $M_{i,j'}$ is unstable by Theorem \ref{thm: 4lines}. Similarly, if $j'>k+2n$ we have $t^0_{\ell+n,k+2n}<c^0<t_{i,j'}$ making $M_{i,j'}$ unstable. Thus the only possible values of $j'=j+pn$ making stable $M_{i,j'}$ are the two with $k<j'<k+2n$.
\end{proof}

\begin{lem}[Finiteness Lemma]\label{finiteness lemma}
Suppose that, on the critical line $t=c^0$, all positive points are above all negative points. I.e., $k\triangleright \ell$ whenever $\vare_k=+$ and $\vare_\ell=-$. Then there are an infinite number of stable modules.
\end{lem}

\begin{proof}
For any positive integer $p$ we consider the (finite) set of all pairs $(i,j)$ with $-pn\le i<j\le pn$ so that $\vare_i=+$, $\vare_j=-$. Since $f_i(c^0)>f_j(c^0)$ the curves $L_i,L_j$ must cross at some point $t_{ij}<c^0$. Choose $(i,j)$ so that $t_{ij}<c^0$ is maximal and, if there is a tie, choose whichever has smaller value of $j-i$. Then we claim that this intersection is stable. To see this consider any $i<k<j$ with $\vare_k=+$. If $f_k(t_{ij})\le f_j(t_{ij})$ then $L_k,L_j$ cross at some point $t_{ij}\le t_{kj}<c^0$. Since $j-k<j-i$ this contradicts the choice of $(i,j)$. Thus $f_k(t_{ij})>f_j(t_{ij})$. Similarly, $f_\ell(t_{ij})<f_i(t_{ij})=f_j(t_{ij})$ for all $i<\ell<j$ with $\vare_\ell=-$. So, $M_{ij}$ is stable by Theorem \ref{thm: stable iff blue above and red below}.

Now increase $p$ by 1. Then $-(p+1)n\le i-n<j+n\le (p+1)n$. Since $t_{ij}<c^0$ we have: 
\[
	f_{i-n}(t_{ij})<f_i(t_{ij})=f_j(t_{ij})<f_{j+n}(t_{ij}).
\]
So, the wires $L_{i-n}$ and $L_{j+n}$ cross at some point $t_{ij}<t_{i-n,j+n}<c^0$. So, there exists another stable $M_{i',j'}$ with $t_{ij}<t_{i'j'}<c^0$. Proceeding in this way, we get an infinite sequence of nonisomorphic stable modules.
\end{proof}

\begin{prop}\label{prop: regular stable modules}
Let $M_{ij}$ be a stable module in a maximal green sequence for the quiver $\widetilde A_{a,b}^\vare$. Suppose $\vare_i=\vare_j$. Then $j-i<n$.
\end{prop}

\begin{proof}
We exclude first the case when $i<j$ are congruent modulo $n$. In that case, $f_i(c^0)=f_j(c^0)$. By the Finiteness Lemma, there must be either a negative point  $\ell\triangleleft i$ or a positive point $k\triangleright i$. By adding multiples of $n$ we can assume $i<k,\ell<j$. Then $M_{ij}$ is unstable by Theorem \ref{thm: stable iff blue above and red below}. In the remaining cases when $i<j$ are not congruent module $n$, by the 2-point Lemma \ref{2pt lem} we must have $j-i<n$.
\end{proof}

\subsection{The upper bound}
	From the above lemmas, we can determine an upper bound for the number of stable intersections when finite: By the Finiteness Lemma there must be at least one positive point below a negative point. We will show that the following pattern of points on the critical line is the one which gives the largest finite number of stable intersections.
	\begin{equation}\label{max pattern of points on critical line}
\begin{array}{rc}
 	\text{$a-1$ positive points}&\begin{cases} +\\
	\vdots\\
	+
    \end{cases}
    \\
 \ell & -\\
 k & +\\
 	\text{$b-1$ negative points}&\begin{cases} -\\
	\vdots\\
	-
    \end{cases}
\end{array}
\end{equation}

\begin{thm}\label{thm M1a}
Maximal green sequences for a quiver of type $\widetilde A_{a,b}$ have length at most
\[
	\binom{a+b}{2} + ab.
\]
In every maximal green sequence of this length, the pattern of points on the critical line must be as given in \eqref{max pattern of points on critical line} above.
\end{thm}

\begin{proof} Considering the points on the critical line, let $f_k(c^0)$ be minimal among all points with $\vare_k=+$ and let $f_\ell(c^0)$ be maximal for $\vare_\ell=-$. By the Finiteness Lemma \ref{finiteness lemma}, $k\triangleleft \ell$. By adding a multiple of $n$ to $\ell$ if necessary, we may assume $k<\ell<k+n$.

Let $A$ denote the set of all $j$ with $k<j<k+n$ so that $k\triangleleft j$. Let $A'$ be the subset of $A$ consisting of those $j$ with $\ell \trianglelefteq j$. Then $A'$ contains $\ell$ and the points above $\ell$ which are all positive. Since $k$ is positive and below $\ell$, $a'=|A'|\le a$.

Similarly, let $B$ denote the set of all $i$ with $\ell-n<i<\ell$ so that $i\triangleleft \ell$ and let $B'$ be the subset of $B$ consisting of those $i$ with $i\trianglelefteq k$. Then $b'=|B'|\le b$. On the critical line $t=c^0$, the set of points corresponding to $A$, resp $A'$, is complementary to the set corresponding to $B'$, resp $B$. So, $|A|=n-b'$ and $|B|=n-a'$. 

Consider all pairs $i,j$ (up to translation by $n$) so that $i\triangleleft j$. Since $f_p(c^0)$ takes $n$ values (given by $p=1,2,\cdots,n$) there are $\binom n2$ such pairs $i,j$.

\underline{Claim (a)}: If $i\in B'$ and $j\in A'$ there are at most 2 values of $p$ so that $M_{i,j+pn}$ is stable, namely $p=0,-1$ (i.e., only $M_{ij}$ and $M_{i,j-n}$ might be stable.) 

\underline{Claim (b)}: If $i\notin B'+n\ZZ$ or $j\notin A'+n\ZZ$, there is at most one value of $p$ so that $M_{i,j+pn}$ is stable.

In other words, the pair $(i,j)$ gives at most one stable module unless $i\in B'$ and $j\in A'$ in which case there might be two stable modules. So, the Claim implies that there are at most $\binom n2+a'b'$ stable modules. This is maximal when $a'=a$ and $b'=b$ which happens exactly when the pattern of points on the critical line is given by \eqref{max pattern of points on critical line}. Thus, these claims imply the Theorem.

 Proof of Claim (b): Suppose $i\notin B'+n\ZZ$. Then $i,j\in A$ (up to translation by a multiple of $n$). By the 3-point Lemma (\ref{3pt lem}), only $M_{ij}$ might be stable for $i,j\in A$. Similarly, if $j\notin A'+n\ZZ$ then we may assume $i,j\in B$ in which case, again, only $M_{ij}$ might be stable.
 
 Proof of Claim (a): Given $i\in B', j\in A'$ we apply the 4-point Lemma \ref{4pt lem} to $\ell<i+n<\ell+n$ and $k<j<k+n$ to give the statement of Claim (a) in the case when $i\neq k$ and $j\neq \ell$. In the special cases when either $i=k$ or $j=\ell$, the statement of Claim (a) is give by the 2-point Lemma \ref{2pt lem}.

 This proves Claims (a) and (b). The Theorem follows.
\end{proof}

Claims (a) and (b) in the above proof also imply the following converse of the Finiteness Lemma. Given a green path $\gamma$ with corresponding functions $f_i$ we say that $(k,\ell)$ is an {\bf essential pair} for $\gamma$ if $\vare_k=+$, $\vare_\ell=-$ and $k\triangleleft\ell$ ($f_k(c^0)<f_\ell(c^0)$).

\begin{thm}[Finiteness Theorem]\label{converse of Finiteness Lemma} There are only finitely many $\gamma$ stable modules if and only if there exists an essential pair for $\gamma$
\end{thm}

\begin{proof} This condition is necessary by the Finiteness Lemma \ref{finiteness lemma}. The converse follows from the proof of Theorem \ref{thm M1a} above since Claims (a) and (b) in the proof imply that there are at most $\binom n2+a'b'$ stable modules where $a'=|A'|,b'=|B'|$ are the sizes of the two finite sets in the proof of Theorem \ref{thm M1a} which exist when there is an essential pair.
\end{proof}

\subsection{Sets of stable modules}

Examining the details of the proof of Theorem \ref{thm M1a}, we obtain the list of stable modules of the maximum length:

\begin{defn}\label{def: Skell}
For any $k<\ell<k+n$ with $\vare_k=+,\vare_\ell=-$, let $A_{k\ell}$, $B_{k\ell}$ denote the following sets of integers.
\[
	A_{k\ell}:=\{\ell\}\cup \{j \,:\, k<j<k+n,\,	\vare_j=+ \}
\]
\[
	B_{k\ell}:=\{k\}\cup\{i\,:\,\ell-n<i<\ell,\, \vare_i=- \}
\]
Note that $|A_{k\ell}|=a$, $|B_{k\ell}|=b$.
Let $\cS_{k\ell}$ denote the set of modules in the following list.
\begin{enumerate}
\item $M_{ij}$ where $i,j\in A_{k\ell}$ with $i<j$
\item $M_{ij}$ where $i,j\in B_{k\ell}$ with $i<j$
\item $M_{ij}$ (or $M_{ji}$ if $i>j$) where $j\in A_{k\ell}$, $i\in B_{k\ell}$
\item $M_{i,j-n}$ (or $M_{j-n,i}$ if $i>j-n$) where $j\in A_{k\ell}$, $i\in B_{k\ell}$ 
\end{enumerate}
\end{defn}

\begin{thm}\label{thm: Skell are the longest sets}\label{thm M2a}
For any maximal green sequence of length $\binom{a+b}{2} + ab$, the set of stable modules is one of the sets $\cS_{k\ell}$ defined above.
\end{thm}

\begin{proof}
The sets $A',B'$ in the proof of Theorem \ref{thm M1a} are equal to $A_{k\ell}, B_{k\ell}$ in the case when the pattern of points on the critical line is as given in \eqref{max pattern of points on critical line}. The possible stable modules of Claim (b) in the proof are listed in (1) and (2) in Definition \ref{def: Skell}. The possible stable modules of Claim (a) in the proof are listed in (3) and (4) in Definition \ref{def: Skell}.
\end{proof}

\begin{lem}\label{lem: Skl are distinct}
If $(a,b)\neq (2,2)$, the sets $\cS_{k\ell}$ are distinct.
\end{lem}

\begin{proof}
We will show that, if $a=1$ or $a\ge 3$, the sets $\cS_{k\ell}$ are distinct, i.e., the set $\cS=\cS_{k\ell}$ determines $k$ and $\ell$ (modulo $n$). The case $b\neq2$ is analogous.

Case 1: $a=1$. Then $k$ is unique modulo $n$. Given $k$, the set $\cS=\cS_{k\ell}$ determines $\ell$ since $\ell$ is the only integer between $k$ and $k+n$ for which $M_{k\ell}$ and $M_{\ell-n,k}$ are both in the set $\cS$. So, the $b$ sets $\cS_{k\ell}$ are distinct.

Case 2: $a\ge3$. Then we claim that the set $\cS=\cS_{k\ell}$ uniquely determines $\ell$ modulo $n$. The reason is that $\ell$ is the unique integer modulo $n$ with $\vare_\ell=-$ which is ``doubly paired'' with only one integer $k$ modulo $n$ with sign $\vare_k=+$ where, by ``doubly paired'', we mean there are two values of $p$ for which $M_{\ell,k+pn}$ are in $\cS$. (The negative points not equal to $\ell$ are paired with the $a-1\ge2$ positive points not equal to $k$.) Furthermore, $\ell$ determines $k$ modulo $n$ since $k$ and $\ell$ are doubly paired.
\end{proof}

In the next section we will show (Theorem \ref{thm M2b}) that each of the sets $\cS_{k\ell}$ occur as the stable set of modules of some maximal green sequence for $\widetilde A^\vare_{a,b}$ for any sign function $\vare$. Together with the above lemma, this implies the following 

\begin{thm}[Theorem M2]\label{thm M2}
For $(a,b)\neq(2,2)$ there are exactly $ab$ distinct sets of stable modules for any maximal green sequence for $\widetilde A^\vare_{a,b}$ with the maximal length $\binom {a+b}2+ab$. 
\end{thm}

The case $(a,b)=(2,2)$ is an easy exercise:

\begin{prop}\label{prop: Skl are not distinct for (2,2)}
For the quivers $\widetilde A_{2,2}^{+-+-}$, resp. $\widetilde A_{2,2}^{++--}$, there are $2$, resp. $3$, possible sets of stable modules for the maximal green sequences of the maximum length which is $10$.
\end{prop}

\begin{proof}
{
For $\widetilde A_{2,2}^{+-+-}$, the possible values of the essential pair $(k,\ell)$ are $(1,2),(1,4),(3,4),(3,6)$. But $A_{14}=\{3,4\}=B_{36}$ and both elements of $A_{36}=\{5,6\}$ are $n=4$ more than the elements of $B_{14}=\{1,2\}$. So, $\cS_{14}=\cS_{36}$. Similarly, $\cS_{12}=\cS_{34}$. So, there are only two distinct sets of the form $\cS_{k\ell}$.
}

{
For $\widetilde A_{2,2}^{++--}$, the formula for $\cS_{k\ell}$ gives: 
\[
\begin{array}{c|cc|cc}
	(k,\ell) & B_{k\ell} & A_{k\ell} & \text{$ij$ with $M_{ij}\in \cS_{k\ell}$}\\
	\hline
	(1,3) & \{1,0\} & \{3,2\} & 01,23,02,03,12,13,24,25,34,35 &=\cS_{24}\\
	(1,4) & \{1,3\} & \{4,2\} & 24,13, 12, 14, 23, 34, 25, 27, 01, 03 \\
	(2,3) & \{2,0\} & \{3,5\} & 02, 35, 03, 05, 23, 25, 34, 36, 01,12\\
	(2,4)& \{2,3\} & \{4,5\} & 23, 01, 24, 25, 34, 35, 02, 03, 12, 13 &=\cS_{13}
\end{array}
\]
Thus $\cS_{14}\neq \cS_{23}$ with neither equal to $\cS_{13}=\cS_{24}$. So, there are three distinct sets $\cS_{k\ell}$.
}
\end{proof}

% subsection

%\newpage
%------------------------------------------------------------------------
%            sub section {Deletion Lemma}
%------------------------------------------------------------------------

\subsection{Deletion Lemma}

One immediate consequence of Theorem \ref{thm: stable iff blue above and red below} and Corollary \ref{cor: stability in terms of chords Vij}, two theorems we had earlier, is the observation that, when a curve $L_k$ is deleted, stable intersections of other curves remain stable. However, previously unstable intersections, possibly an infinite number of them, might become stable. Using Theorem \ref{converse of Finiteness Lemma} we can prevent this from happening.

We need some notation. For $Q$ a quiver of type $\widetilde A_{a,b}$ and $X$ a subset of the arrow set of $Q$ having at most $a+b-2$ elements, we say that $Q'$ is obtained by {\bf collapsing} the arrows in $X$ if $Q'$ is given by {deleting} each arrow $\alpha_i$ (from $i$ to $i+1$ or $i+1$ to $i$) in $X$ and identifying the source and target of $\alpha_i$. This induces an epimorphism 
\[
\overline\pi:Q_0=\{1,2,\cdots,n\}\to Q_0'=\{1,2,\cdots, n-|X|\}
\]
which is the reduction modulo $n-|X|$ of a map $\pi:\ZZ\onto \ZZ$ which is periodic in the sense that $\pi(i+n)=\pi(i)+n-|X|$ and so that $\pi(1)=1$, $\pi(i)\le \pi(j)$ when $i\le j$ and $\pi(i)=\pi(i+1)$ when $\alpha_i\in X$. Given a module $M_{ij}$ for $KQ$, let $\overline\pi(M_{ij})$ be the $KQ'$ module given by
\[
	\overline\pi(M_{ij})=\begin{cases} 0 & \text{if $X$ contains $i$ or $j$} \\
   M_{\pi(i)\pi(j)} & \text{otherwise}
    \end{cases}
\]

\begin{prop}[Deletion Lemma]\label{deletion lemma}
Let $M_1,\cdots,M_m$ be a maximal green sequence for $\Lambda=KQ$ for $Q$ a quiver of type $\widetilde A_{a,b}$. Let $\gamma$ be a corresponding green path. Let $Q'$ be a quiver obtained from $Q$ by collapsing a set of arrows $X$ in $Q$ where $X$ is disjoint from at least one essential pair for $\gamma$. Then there is a maximal green sequence for $KQ'$ containing as a subsequence the nonzero elements of the sequence $\overline\pi(M_1),\cdots,\overline\pi(M_m)$. Furthermore, if the first sequence is linear, so is the second.\qed
\end{prop}

\begin{rem}\label{collapsing Skl to Skl}
It is an easy exercise to show that collapsing arrows in $X$ sends $\cS_{k\ell}$ to $\cS_{k\ell}$ (a different set with the same name) as long as $X$ is disjoint from the essential pair $(k,\ell)$.
\end{rem}

% subsection

%%\newpage
%------------------------------------------------------------------------
%            sub section {Cluster-tilted $D_n$}
%------------------------------------------------------------------------

\subsection{Cluster-tilted $D_n$}

The analysis above also applies to the case when $b=0$. Then the quiver is a single oriented cycle of length $n=a$. Call this quiver $Q_n$. We impose the condition $rad^{n-1}=0$, i.e., the composition of any $n-1$ arrows is zero. The path algebra $KQ_n$ modulo $rad^{n-1}$ is the ``Jacobian algebra'' of a quiver with potential $\Lambda_n=J(Q_n,W)$. This is ``cluster-tilted'' of type $D_n$. For more details, please see the paper \cite{PartII} which was written to explain cluster-tilted algebras of Dynkin type and the problem of finding a maximal green sequence for these algebras.
	
The algebra $\Lambda_n$ has, up to isomorphism, $n(n-1)$ indecomposable modules $M_{ij}$ where $1\le i\le n$ and $i<j<i+n$. Given any green path $\gamma(t)=t\bbb(t)-\aaa(t)$ for $\Lambda_n$ we have associated functions $f_i:\RR\to\RR^n$ given, just as in the case of $\widetilde A_{a,b}$, by Equation \eqref{functions fi}. The critical line $t=c^0$ is defined as before. The Finiteness Lemma does not hold since there are only finitely many stable modules. However Proposition \ref{prop: regular stable modules} holds by definition and the 2-, 3- and 4- point lemmas all hold with the same proofs as before. 

Maximal green sequences of maximal length for $\Lambda_n$ will arise from the collapsing process described in the Deletion Lemma \ref{deletion lemma} using the following observation.

\begin{rem}\label{rem: stable set Sk}
Let $\Lambda=KQ$ be the algebra of type $\widetilde A_{n,1}^\vare$ with sign function $\vare_i=-$ if and only if $i$ is a multiple of $n+1$. Then $Q_n$ is obtained from $Q$ by collapsing the unique edge in $Q$ with negative sign. The collapsing map $\pi$ from the Deletion Lemma sends the set of $\Lambda$-modules $\cS_{k,n+1}$ to the set of $\Lambda_n$-modules
\[
	\cS_k:=\{M_{ij}\,:\, k\le i<j\le k+n, j-i<n\}
\]
which has $\binom n2+n-1$ elements. We observe that the $n$ sets $\cS_k$ are distinct.
\end{rem}

\begin{thm}\label{thm: upper bound for MGS of Dn}
Maximal green sequences for $\Lambda_n=KQ_n/rad^{n-1}$ have length at most
\[
	\binom n2+n-1
\]
and each such maximal green sequence has stable set of modules $\cS_k$ for some $k$.
\end{thm} 

\begin{proof}
Let $f_k(c^0)$ be the lowest point on the critical line. By the 2-point Lemma \ref{2pt lem} and the 3-point Lemma \ref{3pt lem}, the only possible stable modules are the elements of $\cS_k$.
\end{proof}

We give a central charge $Z:K_0\Lambda_n\to\CC$ which realizes each maximal stability set $\cS_k$. Together with Theorem \ref{thm: upper bound for MGS of Dn} above, this will prove Corollary L4 from the introduction. (See Figure \ref{fig: C(Z) under parabola} for the corresponding chord diagram.)

\begin{cor}\label{cor L4, proved}
For every $k$ there is a standard linear stability condition on $\Lambda_n$ with stable set of modules $\cS_k$.
\end{cor}

\begin{proof}
Take $\bbb=(1,1,\cdots,1)$ and let $\aaa=(a_1,\cdots,a_n)$ be given by $a_j=c_j-c_{j-1}$ where $c_j=-(2k+n-2j)^2$.
\end{proof}

% Section:

%\newpage
%%%%%%%%%%%%%%%%%%%%%%%%%%
%
%                Section  {6chordsD}
%
%%%%%%%%%%%%%%%%%%%%%%%%%%

\setcounter{section}5

\section{Chord diagrams for $\widetilde A_{a,b}$}\label{sec6}\label{chord diagrams for tilde A}

In this section we use (periodic) chord diagrams to construct linear stability conditions having the maximum number of stable modules. We also construct piecewise linear green paths having any given maximal set $\cS_{k\ell}$ of stable modules as given in Definition \ref{def: Skell} completing the proof of Theorem M2.

\subsection{Periodic stablility polygon $\widetilde C(Z)$}\label{ss periodic polygon}

The first step is to show that, without loss of generality, we may assume that $m=0$ and $c$, the critical slope of $Z$, is also zero. We have already noted that the value of $m$ is irrelevant. Setting $m=0$ makes:
\[
	f_i(t)=a_1+\cdots+a_i-t(b_1+\cdots+b_i)
\]
We translate all wires $L_i$ to the left by $c$ units by replacing $f_i$ with
\[
	f_i^c(t)=f_i(t+c)=(a_1-cb_1)+\cdots+(a_i-cb_i)-t(b_1+\cdots+b_i).
\]
This corresponds to a new central charge $Z^c$ given by $Z^c(x)=(\aaa-c\bbb)\cdot x+i\bbb\cdot x$. Then $Z^c$ has the same stable modules as does $Z$ with the slopes of all modules decreased by $c$. Thus, without loss of generality, we may replace $Z$ with $Z^c$ and assume that $c=0$. Equivalently, 
\[
	a_1+\cdots+a_n=0.
\]
We say that $Z$ is {\bf normalized} if it has these two properties ($m=c=0$).

As in the finite $A_n$ case we define the {\bf periodic dual vertices} $p_i\in\RR^2$ for $i\in\NN$ by
\[
	p_i=(b_1+\cdots+b_i,a_1+\cdots+a_i)
\]
Since $c=0$, $p_{i+n}=p_i+(B,0)$ for all $i\in\NN$ where $B=b_1+\cdots+b_n$. For $i<0$ we can then define $p_i=p_{i+kn}-k(B,0)$ for sufficiently large $k$.

For all integers $i<j$ the {\bf chord} $V_{ij}$ is the line segment with endpoints $p_i$, $p_j$. As in the finite case (Theorem \ref{thm: stability of chords}) we have the following.

\begin{thm}\label{thm: Z-stability in terms of chords}
For $Z$ a normalized central charge on $K\widetilde A_{ab}^\vare$, a module $M_{ij}$ is $Z$-semistable if and only if the chord $V_{ij}$ has no negative dual vertices above it and no positive vertices below it. $M_{ij}$ is $Z$-stable if, in addition, $V_{ij}$ has no dual vertices in its interior.\qed
\end{thm}

For example, $M_{i,i+2n}$ is never $Z$-stable since $p_{i+n}$ is in the interior of $ V_{i,i+2n}$.

In the following corollary it is essential that $Z$ be normalized. Then $p_i,p_{i+n}$ have the same height $y_i=y_{i+n}$ for all $i\in\ZZ$. So, at most $n$ distinct heights are attained.

\begin{cor}\label{cor: cannot have all pluses above all minuses}
If the height $y_k$ of every positive dual vertex $p_k$ is greater than the height $y_\ell$ of every negative dual vertex $p_\ell$ then there are infinitely many nonisomorphic $Z$-stable modules $M_{ij}$.
\end{cor}

\begin{proof}
Let $p_k$ be a positive dual vertex with minimum height $y_k$. Let $p_\ell$ be a negative dual vertex of maximum height $y_\ell$. Then $y_k>y_\ell=y_{\ell+sn}$ for all integers $s$. Then any interior point in any chord of the form $V_{k,\ell+sn}$ has $y$-coordinate greater than $y_\ell$ and less than $y_k$. By Theorem \ref{thm: Z-stability in terms of chords} the corresponding modules $M_{k,\ell+sn}$ are $Z$-stable for all integers $s$.
\end{proof}

In analogy with the finite case, the {\bf periodic stability polygon} $\widetilde C(Z)$ is defined to be the intersection 
\[
	\widetilde C(Z)=\widetilde C^+(Z)\cap\widetilde C^-(Z)
\]
where $\widetilde C^+(Z),\widetilde C^-(Z)$ are defined as follows. Let $p_{k_i}$ be the positive dual vertices with fixed indexing $k_i, i\in\ZZ$. For every pair of consecutive elements in this set, $p_{k_{i-1}},p_{k_i}$, let $\widetilde C_i^+(Z)$ be the union of the convex hull of the points $p_j$ for $k_{i-1}\le j\le k_i$ with the set of all points below this convex hull, i.e., points with $x$-coordinate equal and $y$-coordinate less than a point in this convex hull. Define $\widetilde C^+(Z)$ to be the union of these sets $\widetilde C_i^+(Z)$. Similarly, $\widetilde C^-(Z)$ is the union of sets $\widetilde C_j^-(Z)$ which are defined to be the convex hull of all dual vertices $p_i$ between and including two consecutive negative vertices union the set of all points above this convex hull. 

Theorem \ref{thm: Z-stability in terms of chords} is equivalent to the following affine analogue of Theorem \ref{thm: Mij Z-semistable iff Vij in C(Z)}.

\begin{thm}\label{thm: Z-stability using periodic stability polygon}
For a normalized central charge $Z$ on $K\widetilde A_{ab}^\vare$, the string module $M_{ij}$ is $Z$-semistable if and only if $V_{ij}\subset \widetilde C(Z)$. $M_{ij}$ is $Z$-stable if, in addition, the chord $V_{ij}$ contains no dual vertex $p_k$ in its interior.\qed
\end{thm}

\subsection{Linearity of maximal stable sets}

\begin{thm}\label{thm: existence of Z} Suppose $k<\ell<k+n$. Then, a sufficient condition for the stable set $\cS_{k\ell}$ to be linear is if, in the sequence of signs $\vare_j$ for $k<j<\ell$, all negative signs come before all positive signs. Another sufficient condition is if, in the sequence of signs $\vare_j$ for $\ell<j<k+n$, all positive signs come before all negative signs.
\end{thm}

\begin{rem}
Theorem \ref{thm: nonexistence of Z} below shows that if neither of these conditions holds, the stable set $\cS_{k\ell}$ is nonlinear. Thus, these conditions, taken together, are necessary and sufficient.
\end{rem}

\begin{proof}
We prove the first statement. The second is analogous. The proof is given by drawing one periodic stability polygon (Figure \ref{Fig: good chord diagram}) which will generate all examples.
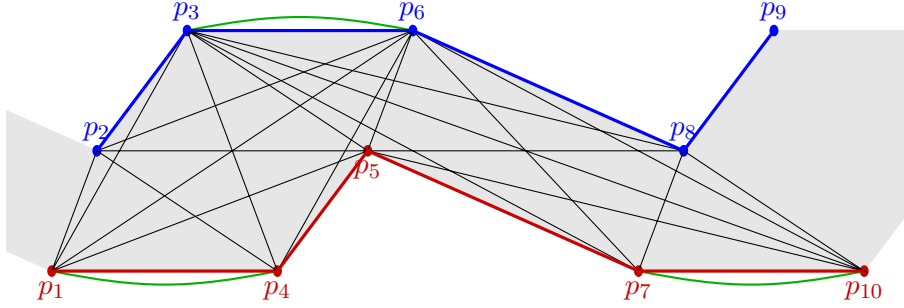
\begin{figure}[htbp]
\begin{center}
\begin{tikzpicture}[xscale=.6,yscale=.8]
\clip (-2,-3) rectangle (18,2.5);
\coordinate (P1a) at (1,-2.3); 
\coordinate (P1b) at (2,-2.3); 
\coordinate (P7a) at (14,-2.3); 
\coordinate (P7b) at (15,-2.3); 
\coordinate (P7) at (12,-2);
\coordinate (P1) at (-1,-2); % used to be called P0
\coordinate (P2) at (0,0);
\coordinate (P8) at (13,0);
\coordinate (P3) at (2,2);
\coordinate (P3a) at (4,2.3);
\coordinate (P3b) at (5,2.3);
\coordinate (P9) at (15,2);
\coordinate (P4) at (4,-2);
\coordinate (P10) at (17,-2);
\coordinate (P5) at (6,0);
\coordinate (P6) at (7,2);
\coordinate (P11) at (11,1.2);
\draw[color=gray!20!white,fill] (P1)--(P2)--(P3)--(P6)--(P5)--(P4)--(P1) (P5)--(P7)--(P8)--(P6)--(P5);
\begin{scope}[xshift=13cm]
\draw[color=gray!20!white,fill] (-1,-2)--(0,0)--(2,2)--(7,2)--(6,0)--(4,-2)--(-1,-2) ;
\end{scope}
\begin{scope}[xshift=-13cm]
\draw[color=gray!20!white,fill]  (6,0)--(12,-2)--(13,0)--(7,2)--(6,0);
\end{scope}
\draw[color=red!80!black,fill] (P1) circle[radius=2.5pt] node[below]{$p_1$};
\draw[color=red!80!black,fill] (P4) circle[radius=2.5pt] node[below]{$p_4$};
\draw[color=red!80!black,fill] (P5) circle[radius=2.5pt] node[below]{$p_5$};
\draw[color=red!80!black,fill] (P7) circle[radius=2.5pt]node[below]{$p_7$};
\draw[color=red!80!black,fill] (P10) circle[radius=2.5pt]node[below]{$p_{10}$};
\draw[color=blue,fill] (P2) circle[radius=2.5pt] node[above]{$p_2$};
\draw[color=blue,fill] (P3) circle[radius=2.5pt] node[above]{$p_3$};
\draw[color=blue,fill] (P6) circle[radius=2.5pt] node[above]{$p_6$};
\draw[color=blue,fill] (P8) circle[radius=2.5pt] node[above]{$p_8$};
\draw[color=blue,fill] (P9) circle[radius=2.5pt] node[above]{$p_9$};
\draw (P1)--(P5)--(P2)--(P1)--(P3)--(P4) (P1)--(P6);
\draw (P6)--(P3)--(P5)--(P6)--(P7) (P6)--(P4)--(P2)--(P6);
\draw (P5)--(P8)--(P10)--(P3)--(P7)--(P8)--(P3) (P5)--(P10)--(P6);
\draw[thick,color=green!70!black] (P1)..controls (P1a) and (P1b)..(P4);
\draw[thick,color=green!70!black] (P7)..controls (P7a) and (P7b)..(P10);
\draw[thick,color=green!70!black] (P3)..controls (P3a) and (P3b)..(P6);
\draw[very thick,color=blue] (P2)--(P3)--(P6)--(P8)--(P9);
\draw[very thick,color=red!80!black] (P1)--(P4)--(P5)--(P7)--(P10);
\end{tikzpicture}
\caption{This is a periodic chord diagram for $\widetilde A_{33}^{-++--+}$ with all 24 chords of $\cS_{25}$ being within the polygon and thus stable. Additional vertices can be placed along the green curves.}
\label{Fig: good chord diagram}
\end{center}
\end{figure}
The points $p_1,\cdots,p_6$ have coordinates $(-1,-2), (0,0), (2,2), (4,\vare-2), (6,\vare), (7,2+\vare), (12,-2)$, etc with $p_{i+6}=p_i+(13,0)$. Then $V_{12},V_{45}$ have slope 2, $V_{23}, V_{45}$ have slope 1, $V_{24},V_{35}$ have slope approximately $-1/2$ and $V_{57},V_{68}$ have slope approximately $-1/3$. In the drawing $\vare=0$. But any positive $\vare<2/5$ will make all 24 chords of $\cS_{25}$ stable.

To draw more general cases, we should insert vertices along the chords $V_{14}$ and $V_{36}$ along a curve (in green in Figure \ref{Fig: good chord diagram}) which is slightly concave up along $V_{14}$ and slightly concave down along $V_{36}$. As long as all the positive points between $\ell$ and $k+n$ come before all the negative points, such a figure will be accurate and make all chords in $\cS_{k\ell}$ stable.
\end{proof}

\begin{thm}\label{thm: nonexistence of Z}
The stable set of modules $\cS_{k\ell}$ is nonlinear if and only if there exist $k<k'<\ell'<\ell<\ell''<k''<k+n$ so that $\vare_{k'},\vare_{k''}$ are positive and $\vare_{\ell'},\vare_{\ell''}$ are negative.
\end{thm}

\begin{proof}
If $k',k'',\ell',\ell''$ do not all exist then the stable set $\cS_{k\ell}$ is linear by Theorem \ref{thm: existence of Z}. Therefore, it suffices to show that, when they do exist, $\cS_{k\ell}$ is nonlinear. 

By the Deletion Lemma \ref{deletion lemma}, we may delete (collapse) all the other edges of the quiver and assume that the quiver is $\widetilde A_{33}^{+++---}$ with $k=2,\ell=5$. Consider the periodic stability polygon $\widetilde C(Z)$ with vertices $p_i$. If we write $p_i=(x_i,y_i)$ then
\[
	y_i=y_{i+6m}=a_1+\cdots+a_i=f_i(0).
\]
By construction of the maximal stability set $\cS_{k\ell}=\cS_{25}$ we therefore have:
\begin{equation}\label{eq: height pattern}
	{\color{red}y_0,y_4}<{\color{blue}y_2}<{\color{red}y_5}<{\color{blue}y_1,y_3}
\end{equation}
But $\cS_{25}$ contains $M_{03},M_{14},M_{36},M_{47}$. Since these are stable, the dual vertex $p_2$ must be above the chords $V_{03},V_{14}$. So, $y_2$ must be greater than the $y$-coordinate of $(s_1,s_2)=V_{03}\cap V_{14}$. Similarly, the $y$-coordinate of $(t_1,t_2)=V_{36}\cap V_{47}$ must be greater than $y_5$. So,
\[
	s_2<y_2<y_5<t_2.
\]
However, examination of Figure \ref{Fig: periodic stability polygon for A6} shows that $t_2<s_2$ which gives a contradiction. The proof is: move $p_1$ (and $p_7$) left until it is right above $p_0$ (and $p_7$ above $p_6$). Then move $p_3=(x_3,y_3)$ right until $x_3=x_4$. Both moves decreases $s_2$ and increase $t_2$. After the moves we have $s_2=t_2$. So, the original values must be related by $s_2>t_2$. This is a contradiction showing that $\cS_{25}$ is not the stability set corresponding to the height pattern \eqref{eq: height pattern}. But the analysis of the critical line shows that this height pattern is the only one which can give this stability set. Therefore $\cS_{25}$ is not the stability set of any central charge.
\end{proof}

\begin{figure}[htbp]
\begin{center}
\begin{tikzpicture}[yscale=.8]
\coordinate (P0) at (0,0.5);
\coordinate (P6) at (6,0.5);
\coordinate (L03) at (.4,1.5);
\coordinate (P1) at (0.4,3.8);
\coordinate (P7) at (6.4,3.8);
\coordinate (L14) at (2.9,1.5);
\coordinate (S) at (1.7,2.35);
\coordinate (P2) at (1.9,3.3);
\coordinate (P8) at (7.9,3.3);
\coordinate (P3) at (3.3,4.2);
\coordinate (P9) at (9.3,4.2);
\coordinate (L36) at (4.4,3.3);
\coordinate (P4) at (3.6,0);
\coordinate (P10) at (9.6,0);
\coordinate (L47) at (6.2,2.8);
\coordinate (T) at (5.1,1.9);
\coordinate (P5) at (5,1.2);
\coordinate (P11) at (11,1.2);
\draw (S)node[right]{$(s_1,s_2)$};
\draw (T)node[right]{$(t_1,t_2)$};
\draw[color=red!80!black,fill] (P0) circle[radius=2.5pt] node[below]{$p_0$};
\draw[color=red!80!black,fill] (P4) circle[radius=2.5pt] node[below]{$p_4$};
\draw[color=red!80!black,fill] (P5) circle[radius=2.5pt] node[below]{$p_5$};
\draw[color=red!80!black,fill] (P6) circle[radius=2.5pt]node[below]{$p_6$};
\draw[color=red!80!black,fill] (P10) circle[radius=2.5pt]node[below]{$p_{10}$};
\draw[color=blue,fill] (P1) circle[radius=2.5pt] node[above]{$p_1$};
\draw[color=blue,fill] (P2) circle[radius=2.5pt] node[above]{$p_2$};
\draw[color=blue,fill] (P3) circle[radius=2.5pt] node[above]{$p_3$};
\draw[color=blue,fill] (P7) circle[radius=2.5pt] node[above]{$p_7$};
\draw[color=blue,fill] (P8) circle[radius=2.5pt] node[above]{$p_8$};
\draw[color=blue,fill] (P9) circle[radius=2.5pt] node[above]{$p_9$};
\draw (L03) node{$V_{03}$};
\draw (L14) node{$V_{14}$};
\draw (L36) node{$V_{36}$};
\draw (L47) node{$V_{47}$};
\draw (P0)--(P3) (P1)--(P4);
\draw (P6)--(P3) (P7)--(P4);
\draw[very thick,color=blue] (P1)--(P2)--(P3)--(P7)--(P8)--(P9);
\draw[very thick,color=red!80!black] (P0)--(P4)--(P5)--(P6)--(P10);
\end{tikzpicture}
\caption{When $M_{03},M_{14}$ are stable, $p_2$ is higher than $(s_1,s_2)=V_{03}\cap V_{14}$. When $M_{36},M_{47}$ are stable, $p_5$ is lower than $(t_1,t_2)=V_{36}\cap V_{47}$. But $t_2<s_2$, contradicting the assumption that $p_2$ is lower than $p_5$.}
\label{Fig: periodic stability polygon for A6}
\end{center}
\end{figure}
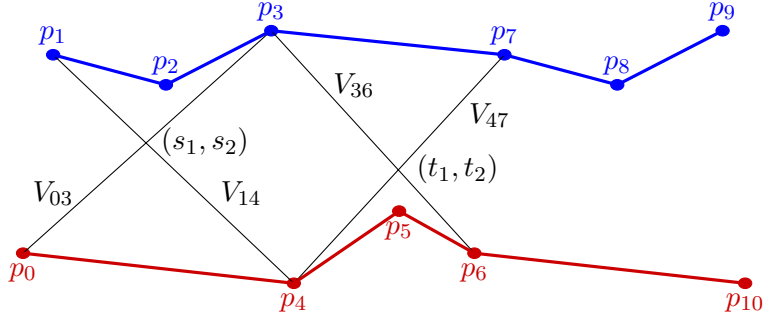

\begin{cor}
If either $a$ or $b$ is $\le 2$, all maximal stable sets of modules are linear.\qed
\end{cor}

\subsection{Nonlinear maximal stable sets}

Let $\cS_{k\ell}$ be one of the maximal stability sets identified as being nonlinear in Theorem \ref{thm: nonexistence of Z}. We will construct a piecewise linear green path $\gamma$ which crosses the walls $D(\beta)$ for $\beta\in \cS_{k\ell}$ and no other walls. This green path will be formed from two linear green paths $\lambda$, $\lambda'$ with the same value at $t=0$: $\lambda(0)=\lambda'(0)$ by
\[
	\gamma(t)=\begin{cases} \lambda(t) & \text{if } t\le 0\\
   \lambda'(t) & \text{if } t>0
    \end{cases}
\]
We say that $\gamma$ is given by {\bf splicing together} $\lambda$ and $\lambda'$ at $t=0$.

Since linear green paths are given by $\lambda(t)=t\bbb-\aaa$, two linear green paths take the same value at $t=0$ if and only if they have the same vector $\aaa\in\RR^n$.

\begin{lem}\label{lem: splicing green paths}
Let $\gamma$ be given by splicing together linear green paths $\lambda_Z,\lambda_{Z'}$ for stability functions $Z,Z'$ with the same vector $\aaa$. Suppose that $Z,Z'$ have no semi-stable modules of slope $0$. Then the set of $\gamma$-stable, resp. $\gamma$-semistable, modules is the union of the following two sets.
\begin{enumerate}
\item The set of $Z$-stable, resp. $Z$-semistable, modules $M$ with negative slope.
\item The set of $Z'$-stable, resp. $Z'$-semistable, modules $N$ with positive slope.
\end{enumerate}
\end{lem}

\begin{rem}
Having the same vector $a$ means that, in the periodic chord diagrams $\widetilde C(Z),\widetilde C(Z')$, the corresponding dual vertices $p_i,p_i'$ have the same $y$-coordinates. The stable/semistable chords of $\gamma$ are the stable/semistable chords of $\widetilde C(Z)$ of negative slope and the stable/semistable chords of $\widetilde C(Z')$ of positive slope.
\end{rem}

\begin{thm}\label{thm M2b}
Given $\Lambda=K\widetilde A_{ab}^\vare$ and $k<\ell<k+n$ with $\vare_k=+,\vare_\ell=-$ there exists a piecewise linear green path $\gamma$, given by splicing together two linear green paths as described above, so that all $\gamma$-semistable modules are stable and the set of $\gamma$-stable modules is $\cS_{k\ell}$.
\end{thm}

\begin{proof}
We will construct the periodic stability polygons $\widetilde C(Z),\widetilde C(Z')$ having the required properties. The functions $Z,Z'$ are determined by the coordinates of the dual vertices $p_i$, $p_i'$ of $\widetilde C(Z),\widetilde C(Z')$, resp. As required, $p_k,p_\ell$ will have $y$-coodinates $y_k<y_\ell$. Then all modules of length $\ge 2n$ will be unstable. So, the set of stable objects is finite and cannot be it greater than $\cS_{k\ell}$. We assume that $p_{k-1},p_{k+1}$ are positive and $p_{\ell-1},p_{\ell+1}$ are negative. There is no loss of generality since, by the Deletion Lemma \ref{deletion lemma}, we can delete these dual vertices and the remaining objects in $\cS_{k\ell}$ will remain stable and give the set $\cS_{k\ell}$ for the smaller quiver.

The dual vertices of $\widetilde C(Z)$ are given as follows (See Figure \ref{Fig: nonlinear periodic stability polygon for Skl}):
\begin{enumerate}
\item $p_k=(-14,-1)$ and $p_{k+n}=(26,-1)$. Thus $p_{i+n}=p_i+(40,0)$ for all $i\in\ZZ$.
\item $p_\ell=(-5,1)$ and $p_{\ell+n}=(35,1)$.
\item All positive dual vertices $p_i$ for $k<i<k+n$ lie on a parabolic curve from $p_{k+1}=(-10,21)$ to $p_{k+n-1}=(11,19)$ which is concave down with slope between 0 and -1 at all points. Also, the $x$-coordinates of these $p_i$ lie in $[-10,-9]\cup [10,11]$.
\item All negative dual vertices $p_j$ for $\ell<j<\ell+n$ lie on a parabolic curve from $p_{\ell+1}=(10,-19)$ to $p_{\ell+n-1}=(31,-21)$ which is concave up with slope between 0 and -1 at all points. Also, the $x$-coordinates of these $p_j$ lie in $[10,11]\cup [30,31]$.
\end{enumerate}
Let $A,B$ be the sets of positive and negative dual vertices listed in (3),(4) above.
The parallel chords $V_{\ell, \ell+1}$ and $V_{k+n-1,k+n}$ have slope $-4/3$ which is $<-1$ and greater than the slopes of $V_{k+1,\ell},V_{k+n,\ell+n-1}$ which are both $-4$. Therefore, all chords between any two elements of the set $A\cup B\cup \{p_\ell,p_{k+n}\}$ are $Z$-stable. These lie in the shaded region of the left hand diagram in Figure \ref{Fig: nonlinear periodic stability polygon for Skl}. All chords of $\cS_{k\ell}$ of negative slope lie either in this set or have the form $V_{ij}$ where $k<i<j<\ell$, $\vare_i=+,\vare_j=-$. These lie in the yellow strip in the left side of Figure \ref{Fig: nonlinear periodic stability polygon for Skl}. Thus, these are also $Z$-stable. Thus, all chords in $\cS_{k\ell}$ with negative slope are $Z$-stable.

The dual vertices of $\widetilde C(Z')$ are the same as those of $\widetilde C(Z')$ except that $p_k$ moves 10 units left and $p_\ell$ moves 10 units right (so $p_k'=(-24,-1)$ and $p_\ell'=(5,1)$). Let $B'$ be the set of negative dual vertices $p_j'=p_j$ with $\ell-n<j<\ell$. Then, all chords between any two points in $A\cup B'\cup \{p_k,p_\ell\}$ lie in the convex shaded region in $\widetilde C(Z')$ in Figure \ref{Fig: nonlinear periodic stability polygon for Skl}. Together with the positive sloped chords $V_{ij}$ for $\ell<i<j<k+n$ with $\vare_i=-,\vare_j=+$, these are all the positive sloped chords in $\cS_{k\ell}$.

Finally, note that no stable chords of $Z$ or $Z'$ are horizontal. By Lemma \ref{lem: splicing green paths}, the green path obtained by splicing $\lambda_Z$ and $\lambda_{Z'}$ makes all elements of $\cS_{k\ell}$ stable.
\end{proof}

\begin{figure}[htbp]
\begin{center}
\begin{tikzpicture}[scale=.9]
\coordinate (P0) at (-3,-1.9);
\coordinate (P4) at (-.9,-2.1);
\coordinate (P04a) at (-2.2,-2.1);
\coordinate (P04b) at (-1.8,-2.1);
\coordinate (P04c) at (-2,-2); %(P04c)node[above]{$B'$}
\coordinate (P6) at (1,-1.9);
\coordinate (P10) at (3.1,-2.1);
\coordinate (P60a) at (1.8,-2.1);
\coordinate (P60b) at (2.2,-2.1);
\coordinate (P60c) at (2,-2); %%%%%%%%%%%
\coordinate (P1) at (-2.9,1.9);
\coordinate (P2) at (-1.4,-.1); % p_k
\coordinate (P2p) at (-2.4,-.1); % p_k'
\coordinate (P8) at (2.6,-.1); % p_{k+n}
\coordinate (P8p) at (1.6,-.1); % p_{k+n}'
\coordinate (P3) at (-1,2.1);
\coordinate (P7) at (1.1,1.9);
\coordinate (P37a) at (-.2,2.1);
\coordinate (P37b) at (.2,2.1);
\coordinate (P37c) at (0,2); %%(P37c)node[below]{$A$}
\coordinate (P3L) at (-1,-2.1);
\coordinate (P9) at (3,2.1);
\coordinate (P4U) at (-.9,2.1);
\coordinate (P5) at (-.5,.1); % p_l
\coordinate (P5p) at (.5,.1); % p+l'
\coordinate (P11) at (11,1.2);
{
{
\draw[fill,color=yellow] (P3)--(P3L)--(P4)--(P4U)--(P3);
\draw (P3)--(P3L) (P4)--(P4U);
\draw[fill,color=gray!20!white] (P3)--(P5)--(P6)..controls (P60a) and (P60b)..(P10)--(P8)--(P7)..controls (P37b) and (P37a)..(P3);
\draw (-3,-.3) node{$\widetilde C(Z)$};
\draw[color=red!80!black,fill] (P0) circle[radius=2.5pt] node[below]{$p_{\ell-n+1}$};
\draw[color=red!80!black,fill] (P4) circle[radius=2.5pt] node[below]{$p_{\ell-1}$};
\draw[color=red!80!black,fill] (P5) circle[radius=2.5pt] node[above]{$p_\ell$};
\draw[color=red!80!black,fill] (P6) circle[radius=2.5pt]node[below]{$p_{\ell+1}$};
\draw[color=red!80!black,fill] (P10) circle[radius=2.5pt]node[below]{$p_{\ell+n-1}$};
\draw[color=blue,fill] (P1) circle[radius=2.5pt] node[above]{$p_{k-1}$};
\draw[color=blue,fill] (P2) circle[radius=2.5pt] node[below]{$p_k$};
\draw[color=blue,fill] (P3) circle[radius=2.5pt] node[above]{$p_{k+1}$};
\draw[color=blue,fill] (P7) circle[radius=2.5pt] node[above]{$p_{k+n-1}$};
\draw[color=blue,fill] (P8) circle[radius=2.5pt] node[right]{$p_{k+n}$};
\draw[color=blue,fill] (P9) circle[radius=2.5pt] node[above]{$p_{k+n+1}$};
\draw[very thick,color=blue] (P1)--(P2)--(P3)..controls (P37a) and (P37b)..(P7)--(P8)--(P9) (P37c)node[below]{$A$};
\draw[very thick,color=red!80!black] (P0)..controls (P04a) and (P04b)..(P4)--(P5)--(P6)..controls (P60a) and (P60b)..(P10) (P60c)node[above]{$B$};
}
\begin{scope}[xshift=8.3cm]
\draw[fill,color=yellow] (1,-1.9)--(1,1.9)--(1.1,1.9)--(1.1,-1.9)--(1,-1.9);
\draw (1,-1.9)--(1,1.9) (1.1,1.9)--(1.1,-1.9);
\draw[fill,color=gray!20!white] (-3,-1.9)--(-2.4,-.1)--(-1,2.1)..controls (-.2,2.1) and (.2,2.1)..(1.1,1.9)--(.5,.1)--(-.9,-2.1)..controls (-1.8,-2.1) and (-2.2,-2.1)..(-3,-1.9);
\draw (3,-.3) node{$\widetilde C(Z')$};
\draw[color=red!80!black,fill] (-3,-1.9) circle[radius=2.5pt] node[below]{$p_{\ell-n+1}$};
\draw[color=red!80!black,fill] (-.9,-2.1) circle[radius=2.5pt] node[below]{$p_{\ell-1}$};
\draw[color=red!80!black,fill] (.5,.1) circle[radius=2.5pt] node[above]{$p_\ell'$};
\draw[color=red!80!black,fill] (1,-1.9) circle[radius=2.5pt]node[below]{$p_{\ell+1}$};
\draw[color=red!80!black,fill] (3.1,-2.1) circle[radius=2.5pt]node[below]{$p_{\ell+n-1}$};
\draw[color=blue,fill] (-2.9,1.9) circle[radius=2.5pt] node[above]{$p_{k-1}$};
\draw[color=blue,fill] (-2.4,-.1) circle[radius=2.5pt] node[left]{$p_k'$};
\draw[color=blue,fill] (-1,2.1) circle[radius=2.5pt] node[above]{$p_{k+1}$};
\draw[color=blue,fill] (1.1,1.9) circle[radius=2.5pt] node[above]{$p_{k+n-1}$};
\draw[color=blue,fill] (1.6,-.1) circle[radius=2.5pt] node[below]{$p_{k+n}'$};
\draw[color=blue,fill] (3,2.1) circle[radius=2.5pt] node[above]{$p_{k+n+1}$};
\draw[very thick,color=blue] (-2.9,1.9)--(-2.4,-.1)--(-1,2.1)..controls (-.2,2.1) and (.2,2.1)..(1.1,1.9)--(1.6,-.1)--(3,2.1) (-.1,2)node[below]{$A$};
\draw[very thick,color=red!80!black] (-3,-1.9)..controls (-2.2,-2.1) and (-1.8,-2.1)..(-.9,-2.1)--(.5,.1)--(1,-1.9)..controls (1.8,-2.1) and (2.2,-2.1)..(3.1,-2.1) (-1.8,-2)node[above]{$B'$};
\end{scope}
}
\end{tikzpicture}
\caption{Since each of the shaded (gray) regions is convex, the $\binom n2$ chords in each shaded region are stable. All chords of negative slope in $\cS_{k\ell}$ are in $\widetilde C(Z)$ either in the shaded region or in the yellow strip, both convex. Similarly, all chords in $\cS_{k\ell}$ of positive slope are in the shaded region or yellow strip in $\widetilde C(Z')$. ``Splicing'' these makes all chords in $\cS_{k\ell}$ stable.}
\label{Fig: nonlinear periodic stability polygon for Skl}
\end{center}
\end{figure}
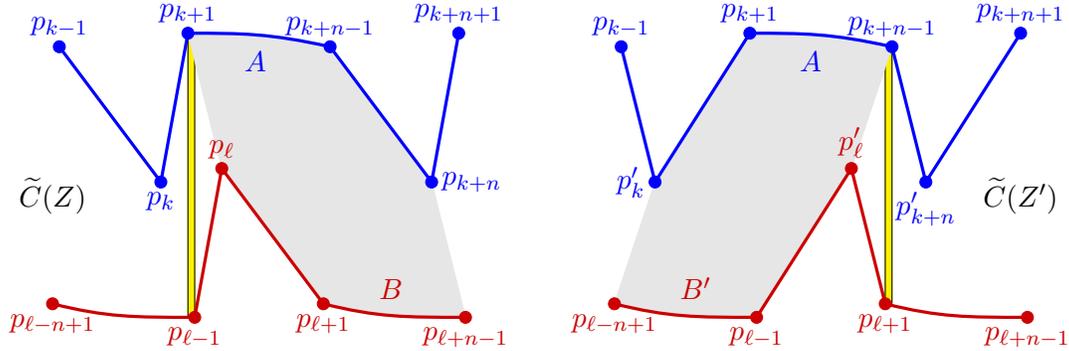

Combining Theorems \ref{thm M2b}, \ref{thm M2a} we get the following.

\begin{cor}\label{cor M2}
Given $\Lambda=K\widetilde A_{ab}^\vare$ and $k<\ell<k+n$ with $\vare_k=+,\vare_\ell=-$ there exists a maximal green sequence with stable module set equal to $\cS_{k\ell}$. This set has $\binom{a+b}2+ab$ elements. These are all the sets of stable modules of a maximal green sequence of this length and there are no maximal green sequences of greater length.
\end{cor}

% subsection

%%\newpage
%------------------------------------------------------------------------
%            sub section Summary
%------------------------------------------------------------------------

%\subsection{Summary}

For the cluster-tilted algebra $\Lambda_n=J(Q_n,W)=KQ_n/rad^{n-1}$, the linear maximal green sequence of Corollary \ref{cor L4, proved} is illustrated in Figure \ref{fig: C(Z) under parabola} in the case $n=5, k=1$.

\begin{figure}[htbp]
\begin{center}
\begin{tikzpicture}[scale=.6]
\coordinate (n3L) at (-3,-3);
\coordinate (n3) at (-3,-.5);
\coordinate (n2) at (-2,-2);
\coordinate (n1) at (-1,-.5); % k=2
\coordinate (c0) at (0,0);
\coordinate (c1) at (1,0);
\coordinate (c2) at (2,-.5); % n=5
\coordinate (c3) at (3,-2);
\coordinate (c4) at (4,-.5);
\coordinate (c5) at (5,0);
\coordinate (c5L) at (5,-3);

\draw[fill,color=gray!20!white] (n3L)--(n3)--(n2)--(n1)--(c0)--(c1)--(c2)--(c3)--(c4)--(c5)--(c5L)--(n3L);
\draw[very thick,color=blue] (n3)--(n2)--(n1)--(c0)--(c1)--(c2)--(c3)--(c4)--(c5);
\foreach\x in {n3,n2,n1,c0,c1,c2,c3,c4,c5} 
\draw[fill] (\x) circle [radius=3pt];
\draw (n3) node[left]{$p_0$};
\draw (n2) node[below]{$p_1$}; % k
\draw (n1) node[above]{$p_2$};
\draw (c0) node[above]{$p_3$};
\draw (c1) node[above]{$p_4$};
\draw (c2) node[above]{$p_5$};
\draw (c3) node[below]{$p_6$}; % n+k
\draw (c4) node[above]{$p_7$};
\draw (c5) node[above]{$p_8$};
\draw (n1)--(c1)--(c3) (n2)--(c0)--(c2) (n2)--(c1) (n1)--(c2) (n2)--(c2) (n1)--(c3) (c0)--(c3);
\end{tikzpicture}
\caption{The $\binom n2+n-1=14$ chords $V_{ij}$ for $1\le i<j\le 6$ are stable.}
\label{fig: C(Z) under parabola}
\end{center}
\end{figure}
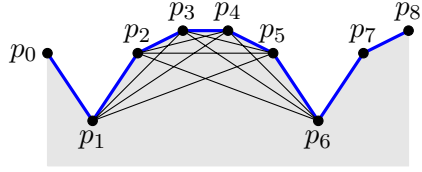
%

%

%\newpage

%%%%%%%%%%%%%%%%%%%%%%%%%%
%
%                Section  Summary
%
%%%%%%%%%%%%%%%%%%%%%%%%%%

\setcounter{section}6

\section{Summary of definitions}\label{sec7}

We summarize the definitions and constructions for \emph{linear stability conditions}. We summarize the changes in the nonlinear case at the end.

\begin{description}
\item[$\bf Z$] The {\bf central charge} (\ref{central charge}) $Z:K_0\Lambda\to \CC$ is given by 
\[
	Z(x)=\aaa\cdot x+i\bbb\cdot x= r(x)e^{i\theta(x)}.
\]
The {\bf standard value} of $\bbb$ is $\bbb=(1,1,\cdots,1)$.
\item[$\bf\sigma_Z$] The {\bf slope} of a module $M$ is
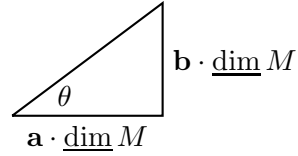
\[
\begin{tikzpicture}%[scale=3]
\coordinate (A) at (0,-.5);
\coordinate (A1) at (0.7,-.5);
\coordinate (AB) at (1,-.5);
\coordinate (B) at (2,-.5);
\coordinate (BC) at (2,.2);
\coordinate (C) at (2,1);
\coordinate (D) at (-4,0);
\begin{scope}%[scale=2]
\draw (D) node{\large$\sigma_Z(M)=\displaystyle\frac{\aaa\cdot \undim M}{\bbb\cdot \undim M}=\cot\theta(M) $};
\end{scope}
	\draw[thick] (A) --(B)--(C)--(A);
	\draw (AB) node[below]{$\aaa\cdot \undim M$};
	\draw (BC) node[right]{$\bbb\cdot \undim M$};
	\draw (A1) node[above]{$\theta$};
\end{tikzpicture}
\]
\item[$\bf\lambda_Z$] The associated {\bf linear green path} $\lambda_Z:\RR\to \RR^n$ is given by
\[
	\lambda_Z(t)=t\bbb-\aaa.
\]
\item[$\bf D(M)$] The {\bf wall} or {\bf semistability set} for a module $M$ (\ref{ss set}) is 
\[
	D(M):=\{
	x\in\RR^n\,:\, x\cdot \undim M=0,\,x\cdot\undim M'\le0\ \text{for all } M'\subseteq M
	\}.
\]
\end{description}

In the special case when $\Lambda=KA_n^\vare$ is the path algebra of a quiver of type $A_n$:%, for example,
\begin{description}
\item[$\bf\varepsilon$] The {\bf sign function} is a mapping $\varepsilon:[0,n]\to \{+,0,-\}$, written $\vare(i)=\vare_i$, so that $\varepsilon_i=0$ iff $i=0$ or $n$.
\item[$\bf A_n^\varepsilon$] The quiver of type $A_n$ with vertices $1,2,\cdots,n$, arrows $i\to i+1$ for $\vare_i=-$ and arrows $i\leftarrow i+1$ when $\vare_i=+$. $\vare_0=\vare_n=0$ are dropped from the notation, e.g.,
\[
	A_3^{-+}:1\to 2\leftarrow 3.
\]
\item[$\bf M_{ij}$] The indecomposable module with support at vertices $i+1,\cdots,j$.
\item[$\bf \beta_{ij}$] $=\undim M_{ij}=e_{i+1}+e_{i+2}+\cdots+e_j$.
\item[$\bf p_i$] The {\bf dual vertices} (\ref{dual vertices}) of the quiver, $p_0,\cdots,p_n$, are
\[
	p_i=(b_1+\cdots+b_i,a_1+\cdots+a_i)\in \RR^2.
\]
\item[$\bf C(Z)$] The {\bf stability polygon} of $Z$ has vertices $p_i$.
\item[$\bf V_{ij}$] The {\bf chords} are $V_{ij}=\overline{p_ip_j}$.
\item[$\bf f_i$]$\RR\to \RR$ are the linear functions (\ref{sec2}) given by
\[
	f_i(t)=t(m-b_1-\cdots-b_i)+a_1+\cdots+a_i
\]
where $m$ is any convenient real number. 
\item[$\bf L_i$] The {\bf wires} $L_i\subseteq\RR^2$ are the graphs of the linear functions $f_i$.
\item[$\bf t_{ij}$] This is the $x$-coordinate of the {\bf intersection point} $L_i\cap L_j$. Also, 
\[
t_{ij}=\sigma_Z(M_{ij})=\text{slope of } V_{ij}.
\]
This is independent of the choice of $m$.
\end{description}

For $\Lambda=\widetilde A_{ab}^\vare$ (\ref{ss affine A}) we have:

\begin{description}
\item[$\bf \widetilde A_{ab}^\vare$] The quiver with vertices $1,2,\cdots,n$, taken to be modulo $n$, with $a$ arrows $k\leftarrow k+1$ for $\vare(k)=+$ and $b$ arrows $\ell\to \ell+1$ for $\vare(\ell)=-$. Equivalently, this is an $n$-periodic quiver with vertex set $\ZZ$.
\item[$\bf Q_n$] $=``\widetilde A_{n0}$" is a single oriented cycle of length $n$ modulo $rad^{n-1}$. I.e., the composition of any $n-1$ arrows is zero. This is cluster-tilted of type $D_n$ for $n\ge4$.
\item[$\bf\vare$] The {\bf periodic sign function} $\vare:\ZZ\to \{+,-\}$ so that $\vare_i=\vare_{i+n}$ and $\vare(k)=+$ for $a$ values of $k$ in $[1,n]$ and $\vare(\ell)=-$ for $b$ values of $\ell$ in $[1,n]$.
\item[$\bf \eta$] The {\bf null root} $\eta=(1,1,\cdots,1)$.
\item[$\bf c$] $\bf=t_{0n}$, the {\bf critical slope} (\ref{ss critical}), is the slope of the null root
\[
	c=\sigma_Z(\eta)=\frac{a_1+\cdots+a_n}{b_1+\cdots+b_n}=\text{slope of } V_{0n}=\text{$x$-coordinate of } L_0\cap L_n
\]
\item[$\bf  p_i, V_{ij},f_i,t_{ij}$] {\bf Periodic} version of these are defined for all $i<j\in\ZZ$ with the same formulas as in the finite $A_n$ case.
\item[$\bf \widetilde C(Z)$] The {\bf periodic stability polygon} (\ref{ss periodic polygon}) has $p_i$ the periodic dual vertices as vertices.
\end{description}

In the nonlinear case we make the following modifications:
\begin{description}
\item[$\bf a$] $\RR\to \RR^n$, ${\bf b:}\RR\to (0,\infty)^n$ are $C^1$ curves with velocity $\aaa'(t)=\bbb'(t)=0$ for $|t|$ large.
\item[$\bf f_i$] $\RR\to \RR$ are the smooth functions given by
\[
	f_i(t)=t(m-b_1(t)-\cdots-b_i(t))+a_1(t)+\cdots+a_i(t)
\]
where $m$ is any convenient fixed real number.
\item[$\bf L_i$] $\subseteq\RR^2$ are the graphs (now smooth curves) of the functions $f_i:\RR\to \RR$.
\item[$\bf T_{ij}$] The set of $x$-coordinates $t_{ij}$ of the elements of $L_i\cap L_j\subset \RR^2$.
\item[$\bf t^0_{ij}, t_{ij}^1$] The minimum and maximum values of $t_{ij}$.
\item[$\bf c^0=t^0_{0n}$] The (smallest) critical slope.
\item[$\bf \cS_{k\ell}$] The nonlinear stability sets of maximum size (\ref{def: Skell}) for any $k<\ell<k+n$ with $\vare_k=+,\vare_\ell=-$ consisting of: 
\begin{enumerate}
\item $M_{ij}$ for all $i,j$ in the set $A$ of size $a$ consisting of $\ell$ and all $k<t<k+n$ with sign $\vare_t=+$
\item $M_{ij}$ for all $i,j$ in the set $B$ of size $b$ consisting of $k$ and all $\ell-n<s<\ell$ with sign $\vare_s=-$
\item $M_{ij}$ and $M_{j,i+n}$ for all $i\in B,j\in A$.
%\item 
\end{enumerate}
\end{description}

The main result of this paper is to show that the sets $\cS_{k\ell}$ are the only stability sets of the maximum size $\binom a2+\binom b2+2ab$ and to determine which are given by linear stability conditions. For $(a,b)\neq (2,2)$ the $ab$ sets $\cS_{k\ell}$ are distinct.

The three kinds of diagrams used in this paper are based on the following theorem proved in the text.

\begin{thm} Let $\Lambda$ be $KA_n^\vare$, $K\widetilde A_{ab}^\vare$ or $KQ_n/rad^{n-1}$. For a rigid indecomposable $\Lambda$-module $M=M_{ij}$ and central charge $Z:K_0\Lambda\to\CC$, the following are equivalent.
\begin{enumerate}
\item[(0)] {\bf definition} (\ref{def: semistable}): $M_{ij}$ is {\bf$Z$-semistable}, i.e., $
	\sigma_Z(M')\ge \sigma_Z(M)$ for all $M'\subseteq M$
\item {\bf wall crossing} (\ref{prop: relation between Z and D(M)}): $\gamma_Z$ goes through $D(M_{ij})$, i.e., $\gamma_Z(t)\in D(M_{ij})$ for some $t\in\RR$.
\item {\bf chord diagram} (\ref{sec1}, \ref{sec6}): For all $i<k,\ell<j$ with $\vare_k=+,\vare_\ell=-$, the point $p_k$ is above or on the chord $V_{ij}$ and the point $p_\ell$ is below or on $V_{ij}$.
\item {\bf wire diagram} (\ref{sec2}, \ref{sec4}): For all $i<k,\ell<j$ with $\vare_k=+,\vare_\ell=-$, 
\[
	f_\ell(t_{ij})\le f_i(t_{ij})= f_j(t_{ij})\le f_k(t_{ij})
\]
i.e., the intersection point $L_i\cap L_j$ is on or below $L_k$ and on or above $L_\ell$.
\end{enumerate}
Furthermore, $t$ in $(1)$ is equal to $\sigma_Z(M_{ij})=t_{ij}$ in $(3)$ is equal to the slope of $V_{ij}$ in $(2)$.
\end{thm}

The relation between chord diagrams and mixed cobinary trees is explained in \cite[Sec 8]{IO}. This will be extended to periodic trees in a revised version of \cite{ITW}. We also expect that $m$-noncrossing trees \cite{mtrees} and periodic versions of such trees can be used to find the maximal lengths of $m$-maximal green sequences using \cite{IZ}.% and \cite{KI}.

%%%%%%%%%%%%%%%%%%%%%%%%%%
%
%                Section  thebibliography
%
%%%%%%%%%%%%%%%%%%%%%%%%%%

%


\begin{thebibliography}{aa}

\bibitem{ASS} Ibrahim Assem, Daniel Simson and Andrzej Skowro\'nski, Elements of the representation theory of associative algebras, 1: Techniques of representation theory, London Mathematical Society Student Texts, vol. 65 (Cambridge University Press, Cambridge, 2006). % cited 1 times


\bibitem{B}Tom Bridgeland, \emph{Stability conditions on triangulated categories}, Ann. Math. \textbf{166}, No. 2 (Sep., 2007), pp. 317--345. % cited 1 times

\bibitem{B2} Tom Bridgeland, \emph{Spaces of stability conditions}, Algebraic geometry--Seattle 2005. Part 1 (2009): 1--21. % cited 1 times

\bibitem{BST} Thomas Br\"ustle, David Smith, Hipolito Treffinger, \emph{Stability conditions, tau-tilting theory and maximal green sequences}, arXiv:1705.08227. % cited 3 times



%\bibitem{BMRRT}Aslak~Bakke Buan, Robert~J. Marsh, Markus Reineke, Idun Reiten, and Gordana Todorov, \emph{Tilting theory and cluster combinatorics}, Adv. Math. \textbf{204} (2006), no.~2, 572--618. 

\bibitem{DW}
Harm Derksen and Jerzy Weyman, \emph{Semi-invariants of quivers and saturation for {L}ittlewood-{R}ichardson coefficients}, J. Amer. Math. Soc. \textbf{13} (2000), no.~3, 467--479 (electronic). % cited once

\bibitem{Keller} {Bernhard Keller}, \emph{On cluster theory and quantum dilogarithm identities}, arXiv1102.4148. %cited 1 times



%\bibitem{DW2}Harm Derksen and Jerzy Weyman, \emph{On the canonical decomposition of quiver representations}, Compositio Mathematica, September 2002, Volume 133, Issue 3, 245--265. % cited once


%\bibitem{FZ}Sergey Fomin and Andrei Zelevinsky, \emph{Cluster algebras. {IV}. {C}oefficients}, Compos. Math. \textbf{143} (2007), no.~1, 112--164. % cited once


\bibitem{PartI} Kiyoshi Igusa, \emph{Linearity of stability conditions}, arXiv:1706.06986. % cited 7 times

\bibitem{PartII} \bysame, \emph{Maximal green sequences for cluster-tilted algebras of finite type}, arXiv:1706.06503. % cited 5 times

\bibitem{mtrees} \bysame, \emph{m-noncrossing trees}, Journal of Algebra and Its Applications (2017): 1850187.% cited once


%\bibitem{IOTW} Kiyoshi Igusa, Kent Orr, Gordana Todorov and Jerzy Weyman, \emph{Cluster complexes via semi-invariants}, Compos. Math. 145 (2009), no. 4, 1001--1034. % cited twice



\bibitem{IOTW2}
  Kiyoshi Igusa, Kent Orr, Gordana Todorov, and Jerzy Weyman, \emph{Modulated semi-invariants}, arXiv: 1507.03051v2 % cited once
  
\bibitem{IO}
Kiyoshi Igusa and Jonah Ostroff, \emph{Mixed cobinary trees}, Journal of Algebra and Its Applications (2017),
https://doi.org/10.1142/S0219498818501700  % cited once
  
\bibitem{ITW} Kiyoshi Igusa, Gordana Todorov and Jerzy Weyman, \emph{Periodic trees and semi-invariants}, arXiv: 1407.0619. % cited once

\bibitem{IZ} Kiyoshi Igusa and Ying Zhou, \emph{Tame hereditary algebras have finitely many $m$-maximal green sequences}, arXiv: 1706.09118. % cited once

%\bibitem{InTh} Colin Ingalls and Hugh Thomas, \emph{Noncrossing partitions and representations of quivers}, Compos. Math. {\bf145} (2009), no. 6, 1533--1562. % cited 8 times

\bibitem{Kase} Ryoichi Kase, \emph{Remarks on lengths of maximal green sequences for quivers of type $\widetilde A_{n,1}$}, arXiv:1507.02852. % cited once

%\bibitem{King} Alastair D. King, \emph{Moduli of representations of finite dimensional algebras}, Quart. J. Math. Oxford (2), \textbf{45} (1994), 515--530.

%\bibitem{KQ}Alastair King and Yu Qiu, \emph{Exchange graphs and Ext quivers}, Adv. Math. {\bf285} (2015), 1106--1154. % cited once


\bibitem{Q}Yu Qiu, \emph{Stability conditions and quantum dilogarithm identities for Dynkin quivers}, Adv. Math. {\bf269} (2015), 220--264. % cited once

%\bibitem{Reading} Nathan Reading, \emph{Universal geometric cluster algebras}, Reading, Nathan. "Universal geometric cluster algebras." Mathematische Zeitschrift 277.1-2 (2014): 499-547.


\bibitem{R} Markus Reineke, \emph{The Harder-Narasimhan system in quantum groups and cohomology of quiver moduli}, Invent. Math. {\bf152} (2003), no. 2, 349--368. % cited once

 
%\bibitem{Woolf} Jonathan Woolf, \emph{Stability conditions, torsion theories and tilting}, J. Lond. Math. Soc. (2) 82 (2010), no. 3, 663--682. % cited twice



\end{thebibliography}
\end{document}
